\documentclass[11pt]{article}
\usepackage{amsfonts,amsmath,amsthm}

\usepackage[usenames,dvipsnames]{color}
\usepackage[colorlinks,%
linkcolor=BrickRed,%
filecolor=BrickRed,%
citecolor=RoyalPurple,%
]{hyperref} 

\def\transpose{{\hbox{\tiny\it T}}}

\newcommand{\RL}{{\mathbb R}}

\newcommand{\IND}{{\mathbb I}}

\def\be{\begin{eqnarray}}
\def\ee{\end{eqnarray}}
\def\ben{\begin{eqnarray*}}
\def\een{\end{eqnarray*}}

\def\barxi{\oo\xi}

\def\ddt{\frac{d}{dt}}

\def\sq{$\Box$}

\def\qed{\ifmmode\sq\else{\unskip\nobreak\hfil
\penalty50\hskip1em\null\nobreak\hfil\sq
\parfillskip=0pt\finalhyphendemerits=0\endgraf}\fi\par\medbreak}

\newsavebox{\junk}
\savebox{\junk}[1.6mm]{\hbox{$|\!|\!|$}}
\def\lll{{\usebox{\junk}}}

\def\state{{\sf X}}
\def\ystate{{\sf Y}}

\newcommand{\field}[1]{\mathbb{#1}}

\def\Re{\field{R}}

\def\Co{\field{C}}

\def\ind{\field{I}}

\def\One{\hbox{\large\bf 1}}

\newcommand{\barR}{{\overline{R}}}

\def\bfmM{{\mbox{\protect\boldmath$M$}}}

\def\bfPhi{\mbox{\protect\boldmath$\Phi$}}
\def\bfPsi{\mbox{\protect\boldmath$\Psi$}}

\def\til={{\widetilde =}}

\def\clB{{\cal B}}

\def\clD{{\cal D}}
\def\clE{{\cal E}}

\def\clH{{\cal H}}

\def\clS{{\cal S}}

\def\half{{\mathchoice{\textstyle \frac{1}{2}}%
{\frac{1}{2}}%
{\hbox{\tiny $\frac{1}{2}$}}%
{\hbox{\tiny $\frac{1}{2}$}} }}

\def\eqdef{\mathbin{:=}}

\def\Prob{{\sf P}}

\def\Expect{{\sf E}}

\def\epsy{\varepsilon}

\def\varble{\,\cdot\,}

\newtheorem{theorem}{Theorem}[section]
\newtheorem{corollary}[theorem]{Corollary}
\newtheorem{proposition}[theorem]{Proposition}
\newtheorem{lemma}[theorem]{Lemma}

\def\Lemma#1{Lemma~\ref{t:#1}}
\def\Proposition#1{Proposition~\ref{t:#1}}
\def\Theorem#1{Theorem~\ref{t:#1}}

\def\bdes{\begin{description}}
\def\edes{\end{description}}

\newcommand{\oo}{\overline}

\def\trace{\mathrm{ trace\,}}

\def\Lv{L_\infty^v}

\def\FRAC#1#2#3{\genfrac{}{}{}{#1}{#2}{#3}}

\def\half{{\mathchoice{\FRAC{1}{1}{2}}%
{\FRAC{1}{1}{2}}%
{\FRAC{3}{1}{2}}%
{\FRAC{3}{1}{2}}}}

\def\fourth{{\mathchoice{\FRAC{1}{1}{4}}%
{\FRAC{1}{1}{4}}%
{\FRAC{3}{1}{4}}%
{\FRAC{3}{1}{4}}}}

\def\eqdist{\buildrel{\rm dist}\over =}

\def\bfmath#1{{\mathchoice{\mbox{\boldmath$#1$}}%
{\mbox{\boldmath$#1$}}%
{\mbox{\boldmath$\scriptstyle#1$}}%
{\mbox{\boldmath$\scriptscriptstyle#1$}}}}

\def\bfmB{\bfmath{B}}

\def\bfmI{\bfmath{I}}

\def\bfUpsilon{\bfmath{\Upsilon}}

\def\transpose{{\hbox{\tiny\it T}}}

\def\bfmI{{\mbox{\protect\boldmath$I$}}}

\def\One{\mbox{\rm{\large{1}}}}

\def\muLeb{\mathop{\mu^{\hbox{\tiny Leb}}}}

\newcounter{rmnum}
\newenvironment{romannum}{\begin{list}{{\upshape (\roman{rmnum})}}{\usecounter{rmnum}
\setlength{\leftmargin}{24pt}
\setlength{\rightmargin}{16pt}
\setlength{\itemindent}{-1pt}
}}{\end{list}}
 
\newcounter{anum}
\newenvironment{alphanum}{\begin{list}{{\upshape (\alph{anum})}}{\usecounter{anum}
\setlength{\leftmargin}{28pt}
\setlength{\rightmargin}{16pt}
\setlength{\itemindent}{0pt}
}}{\end{list}}

\newlength{\noteWidth}
\long\def\notes#1{\ifinner
             {\tiny #1}
             \else
             \marginpar{\parbox[t]{\noteWidth}{\raggedright\tiny #1}}
             \fi}
             
             \def\notes#1{}
\def\archive#1{}

\newcounter{tasks}

{\begin{quote}%
\begin{list}{\hspace{-.75cm}\hbox{\textrm{\rm (\roman{tasks})}\ }}{\usecounter{tasks}%
        \setlength{\labelsep}{0pt}
        \setlength{\leftmargin}{0pt}
        \setlength{\rightmargin}{0pt}
        \setlength{\labelwidth}{0pt}
        \setlength{\listparindent}{0pt}}}%
{\end{list}\end{quote}}

\newcounter{tasksA}

{\begin{quote}%
\begin{list}{\hspace{-.75cm}\hbox{\textrm{\rm (\alph{tasksA})}\ }}{\usecounter{tasksA}%
        \setlength{\labelsep}{0pt}
        \setlength{\leftmargin}{0pt}
        \setlength{\rightmargin}{0pt}
        \setlength{\labelwidth}{0pt}
        \setlength{\listparindent}{0pt}}}%
{\end{list}\end{quote}}
 
\begin{document}
 
\title{\vspace{-2cm}%
Approximating a Diffusion\\ by a Finite-State Hidden Markov Model}
\author{
      	I. Kontoyiannis\thanks{{\bf Corresponding author}. 
		Department of Informatics,
		Athens University of Economics and Business,
		Patission 76, Athens 10434, Greece.
                Email: {\tt yiannis@aueb.gr}.
		I.K.\ was supported 
	by the European Union and Greek National 
	Funds through the Operational Program Education 
	and Lifelong Learning of the National Strategic Reference 
	Framework through the Research Funding Program Thales-Investing 
	in Knowledge Society through the European Social Fund.  
	}
\and
        S.P. Meyn\thanks{Department of Electrical and Computer 
                Engineering,
                University of Florida, Gainesville, USA.
                 Email: {\tt meyn@ece.ufl.edu}.
                S.P.M.\ was supported in part by the National Science Foundation
  ECS-0523620,  and AFOSR grant FA9550-09-1-0190.
  Any opinions, findings,
  and conclusions or recommendations expressed in this material are
  those of the authors and do not necessarily reflect the views of the
  National Science Foundation or AFOSR.}
}

\maketitle
\thispagestyle{empty}
\setcounter{page}{0}
 
\begin{abstract}
For a wide class of continuous-time Markov processes
evolving on an open, connected subset of $\RL^d$, 
the following are shown to be equivalent:
\begin{itemize}
\item[(i)]
The process satisfies (a slightly weaker version of)
the classical Donsker-Varadhan conditions;
\item[(ii)]
The transition semigroup of the process can 
be approximated by a finite-state hidden Markov model,
in a strong sense in terms 
of an associated operator norm;
\item[(iii)]
The resolvent kernel of the process is
`$v$-separable', that is, it can be approximated
arbitrarily well in operator norm by finite-rank
kernels.
\end{itemize}
Under any (hence all) of the above conditions,
the Markov process is shown to have a purely
discrete spectrum on a naturally associated weighted
$L_\infty$ space.

\bigskip

\bigskip


\bigskip


\end{abstract}

\thispagestyle{empty}

\newpage
 
\section{Introduction} 

Consider a continuous-time Markov process 
$\bfPhi = \{\Phi(t) : t\geq 0\}$ taking
values in an open, connected subset $\state$ of 
$\Re^d$, equipped with its associated Borel 
$\sigma$-field $\clB$. We begin by assuming
that $\bfPhi$ is a diffusion; 
that is, it is the solution of the stochastic 
differential equation,
\be
d\Phi(t)=u(\Phi(t))dt +M(\Phi(t))dB(t),
\;\;\;\;t\geq 0,\;\Phi(0)=x,
\label{eq:SDE}
\ee
where $u=(u_1,u_2,\ldots,u_d)^\transpose:\state\to\RL^d$ and 
$M:\state\to\RL^d\times\RL^k$ are locally Lipschitz,
and $\bfmB=\{B(t) : t\geq 0\}$ is $k$-dimensional
standard Brownian motion. [Extensions 
to more general Markov processes are briefly
discussed in Section~\ref{s:extend}.]
Unless explicitly
stated otherwise, throughout the paper
we assume that:
$$
\left. 
\mbox{\parbox{.75\hsize}{\raggedright
The strong Markov process $\bfPhi$ is the unique
	strong solution of (\ref{eq:SDE})
	with continuous sample paths.}}
\right\}
\eqno{\hbox{\bf (A1)}}
$$
The distribution of the process $\bfPhi$ is 
described by the initial condition $\Phi(0)=x\in\state$
and the transition semigroup $\{P^t\}$:
For any $t\ge 0$, $x\in\state$,  $A\in \clB$,
$$
P^t(x,A):=\Prob_x\{\Phi(t)\in A\}:=\Pr\{\Phi(t)\in A\,|\,\Phi(0)=x\}.
$$

Recall that the kernel $P^t$
acts as a linear operator on functions $f:\state\to\RL$ on the
right and on signed measures $\nu$ on $(\state,\clB)$ on
the left, respectively, as,
\archive{minor comment 2 asks us to clarify this, so I delete:, or any kernel $P$ on $\state$,
\\
also, I have a good reason to hold off on capital $F$!}
$$
P^tf\,(x)=\int f(y)P^t(x,dy),
\;\;\;\;
\nu P^t\, (A) = \int \nu(dx)P^t(x,A),
\;\;x\in\state,\,A\in\clB,
$$
whenever the above integrals exist.  Also, for any
signed measure $\nu$ on $(\state,\clB)$ and any
function $f:\state\to\RL$ we write $\nu(f):=\int f d\nu$,
whenever the integral exists.
In this paper we will constrain the domain of functions $f$ to a Banach space defined with respect to a weighted $L_\infty$ norm. 
\archive{to begin to address concern over domain in  math comment 4}

One of the central assumptions we make throughout the paper
is the following regularity condition on the semigroup:
\archive{save this: 
wild indeed!!!
-y
On 5/26/2013 6:23 PM, Sean Meyn wrote:
Thank you - I feel like a complete slob, but all is o.k.
It is so funny that I didn't catch this.  We know,
DR  = R-I
or
I = R-DR
If R has a smooth density, then so does I  !!
Conclusion: R NEVER EVER has a smooth density, for any Markov model.
Wild.
 -Sean}
$$
\left. 
\mbox{\parbox{.75\hsize}{\raggedright
The transition semigroup admits  a continuous density:
There is a continuous function $p$ on $(0,\infty)\times\state\times \state$ such that,
\[
P^t(x,A) = \int_A p(t,x,y)\, dy\,,\qquad x\in\state,\ A\in\clB.
\] }}
\right\}
\eqno{\hbox{\bf (A2)}}
$$
H\"ormander's theorem  \cite[Thm.~38.16]{rogers-williams:II} gives 
 sufficient conditions for (A2). 
Explicit bounds on the density are also available;
see \cite{polcinmen12} and its references.

\subsection{Irreducibility, drift, and semigroup approximations} 

The ergodic theory of continuous-time Markov processes 
is often most easily addressed by translating results 
from the discrete-time domain.  This is achieved,
e.g., in \cite{dowmeytwe95a,meytwe93a,meytwe93b,meytwe93e} 
through consideration of the Markov chain whose transition 
kernel is defined by one of the {\em resolvent kernels}
of $\bfPhi$, defined as, 
\begin{equation}
R_\alpha \eqdef   \int_0^\infty  e^{-\alpha t} P^t\, dt,
\;\;\;\;\alpha>0.
\label{resolvent}
\end{equation} 
In the case $\alpha=1$ we simply
write $R:=R_1 = \int_0^\infty  e^{-  t} P^t\, dt$,
and call $R$ ``the'' resolvent kernel
of the process $\bfPhi$.

The family of resolvent kernels
$\{R_\alpha\}$ is simply the Laplace transform of the semigroup, 
so that each $R_\alpha$ admits a density under (A2).     
This density will not be continuous in general, so we will 
truncate to obtain the positive kernel,
\begin{equation}
\barR_\alpha = \int_{t_0}^{t_1}   e^{-\alpha t} P^t\, dt \,,
\label{e:Rhat}
\end{equation}
where $0<t_0<t_1<\infty$ will  be chosen so that 
$\barR_\alpha$ is a good approximation to $R_\alpha$.   
The approximation admits a continuous density under (A2),
\begin{equation}
\barR_\alpha(x,A) =   \int_A \barxi_\alpha(x,y)\, dy
\,,\qquad x\in\state,\ A\in\clB,
\label{e:har}
\end{equation}
where for each $x,y$,
\[
\barxi_\alpha(x,y) =  \int_{t_0}^{t_1}  e^{-\alpha t} p(t,x,y)\, dt.
\]

\begin{proposition}
\label{t:StrongFeller}
Under Assumptions~{\em (A1)} 
and~{\em (A2)}, for any $\alpha>0$,
the resolvent kernel $R_\alpha$ 
has the strong Feller property.  Moreover,  there exist
continuous functions $s_\alpha, n_\alpha\colon\state\to \Re_+$ that 
are not identically zero, and satisfy,
\begin{equation}
R_\alpha (x, dy)  \ge s_\alpha(x) n_\alpha(y)  \, dy, \qquad x,y\in\state.
  \label{e:Rsmall}
\end{equation}
\end{proposition}

\begin{proof} 
Condition~(A2)  implies the strong Feller property for the semigroup
$\{P^t\}$,
that is, 
the function $P^tf$  is
continuous whenever $f$ is measurable and bounded, 
for $t>0$.  It is then straightforward to show that the kernel 
$R_\alpha$ also has the strong Feller property for 
any $0<\alpha<\infty$.
\archive{Feller here to address minor comment 1}

The existence of the functions
$s_\alpha$ and $n_\alpha$ in the 
lower bound follows from the obvious bound 
$R_\alpha\ge \barR_\alpha$. 
\end{proof}

The function $s_\alpha$ and 
the positive measure defined by 
$\mu_\alpha(dy) = n_\alpha(y) dy$ are called \textit{small}, 
and the inequality  \eqref{e:Rsmall}
 is written in terms of an
outer product as,
$R_\alpha  \ge s_\alpha\otimes\mu_\alpha$;
cf.\ \cite{nummelin:book,meyn-tweedie:book2}.
Without loss of generality
(through normalization)
we always assume that 
$\mu_\alpha(\state)=1$, so that $\mu_\alpha$ defines a probability 
measure on $(\state,\clB)$.  

Some of the results on ergodic theory require the 
following `reachability' condition for $\bfPhi$;
it is a mild irreducibility assumption:
$$
\left. 
\mbox{\parbox{.75\hsize}{\raggedright
	There is a state $x_0\in\state$ such that,
	for any $x\in\state$ and any open set $O$
	containing $x_0$, we have,
	$$
	P^t(x,O)>0,\qquad \hbox{for all $t\geq 0$ sufficiently large.}
	$$
	}}
\right\}
\eqno{\hbox{\bf (A3)}}
$$
Under (A3) we are assured of a single communicating class,
since then the process is $\psi$-irreducible
and aperiodic
with $\psi(\varble)\eqdef R (x_0,\varble)$:
For all $x\in\state$ 
and all $A\in\clB$ such that $R(x_0,A)>0$,
we have,
$$P^t(x,A)>0, \;\;\;\;\;\;\mbox{for all $t$
sufficiently large.}$$
See 
\cite[Theorem~3.3]{meytwe93a} and \Proposition{hypoB} below.

Recall that the {\em generator} of $\bfPhi$
is expressed, for bounded $C^2$ functions 
$f\colon\state\to\RL$, as,  
\begin{equation}
\clD f \, (x) = \sum_i  u_i(x) \frac{d}{\, dx_i} f\, (x)
        +
        \half \sum_{ij}
                \Sigma_{ij}(x) \frac{d^2}{\, dx_i \, dx_j} f\, (x),
		\qquad x\in\state,
  \label{e:diffGen}
\end{equation}
or, in more compact notation,
\[
\clD= u\cdot\nabla + \half \trace(\Sigma \nabla^2) ,
\]
where $\Sigma=MM^\transpose$.
Rather than restricting attention to $C^2$ functions, we consider the extended generator, as in our previous work \cite{kontoyiannis-meyn:II,meytwe93a}.   The function  $f\colon\state\to\Re$ is in the domain of $\clD$ 
if there exists a function $g\colon\state\to\Re$ such that the stochastic process defined by,
\begin{equation}
M(t) =  f(\Phi(t)) - \int_0^t g(\Phi(s))\, ds     ,\qquad t\geq 0,
\label{e:extgenMart}
\end{equation}
is a \textit{local martingale}, for each initial condition $\Phi(0)$ \cite{ethkur86,rogers-williams:II}.
We then write $g=\clD f$.

If $\bfmM$ is in fact a martingale, then the following integral equation holds:
\begin{equation}
 P^t f     = f +\int_0^t P^s g\, ds     ,\qquad t\geq 0\,.
\label{e:extgenMartA}
\end{equation}
See \Proposition{ResolveDR}  for a class of functions $(f,g)$ solving \eqref{e:extgenMartA}.

Fleming's \textit{nonlinear generator} \cite{fle78a} 
for the continuous-time Markov process $\bfPhi$ 
is defined via,
\begin{equation}
\clH(F)\eqdef e^{-F} \clD e^F   \, .
\label{fle78a}
\end{equation}
Its domain is the set of functions $F$ for which $f= e^F$ is
in the domain of $\clD$.   
Theory surrounding multiplicative ergodic theory and large 
deviations based on the nonlinear generator is described,
e.g., in 
\cite{fen99a,wu:01,feng-kurtz:book,kontoyiannis-meyn:II}. We say that the {\em Lyapunov drift criterion} (DV3)  
{\em holds with respect to the Lyapunov function 
$V:\state\to(0,\infty]$}, if
		there exist a function 
		$W\colon\state\to [1,\infty)$, 
                a compact set $C\subset\state$,
                and constants $\delta>0$, $b<\infty$,
such that,
$$
\clH(V)\leq -\delta W+b\ind_C\, .
\eqno{\hbox{(DV3)}}
$$
In most of the subsequent results, the following
strengthened version of (DV3) is assumed:
$$
\left. 
\mbox{\parbox{.75\hsize}{\raggedright
	Condition (DV3) holds with respect to
	continuous functions $V,W$ that have
	compact sublevel sets. 
	}}
\right\}
\eqno{\hbox{\bf (A4)}}
$$
Recall that the sublevel 
sets of a function $F\colon\state\to\Re_+$ are defined by,
\begin{equation}
C_F(r) =\{ x\in\state : F(x) \le r\},\qquad r\ge 0. 
\label{CFr}
\end{equation} 

Note that the local Lipschitz assumption 
in (\ref{eq:SDE}) together with (DV3)
imply (A1); 
namely, that (\ref{eq:SDE}) 
has a unique strong solution $\bfPhi$ 
with continuous sample paths;
see \cite[Theorem~2.1]{meytwe93b} and
\cite[Theorem~11.2]{rogers-williams:II}.

Conditions (A1--A4) are essentially equivalent to
(but weaker than) the conditions imposed 
by Donsker and Varadhan in their pioneering 
work \cite{donsker-varadhan:I-II,donsker-varadhan:III,donsker-varadhan:IV}.
Condition (DV3) is a generalization of the drift condition 
of  Donsker and Varadhan.  Variants of this drift condition are used 
in \cite{balaji-meyn,wu:01,reytho01a,kontoyiannis-meyn:I,guileowuyao09},  
and (DV3) is the central assumption in \cite{kontoyiannis-meyn:II}.

One important application of (DV3) here and in 
\cite{kontoyiannis-meyn:II} is in the truncation of the state 
space -- this is how we obtain a hidden Markov model (HMM)
approximation, where the 
approximating process eventually evolves on a compact set.  
Important related results have been obtained by Wu; see
\cite{wu95a,wu:01,wu04} and the references therein.
Wu, beginning with his 1995 work \cite{wu95a}, has developed 
a similar truncation 
technique for establishing large deviations limit theorems,  
as well as the 
existence of a spectral gap in the $L_p$ norm, in  a spirit similar to this 
paper and \cite{kontoyiannis-meyn:II}.  For bibliographies on 
these methods and other applications see \cite{gonwu06a,guileowuyao09}.    
A significant further contribution of the present paper,
in contrast to the 
earlier work mentioned, is the introduction of the {\em weighted $L_\infty$ 
norm} for applications to large deviations theory and spectral theory.  
In particular, for {\em non-reversible} Markov processes,  	
the theory is greatly simplified and extended by posing spectral theory 
within the weighted $L_\infty$ framework.   



The weighted norm is based on the Lyapunov function $V$ from (DV3).
We let $v=e^V$ and define, for any measurable function 
$g\colon\state\to\Re$, 
\[
\|g\|_{v}\eqdef \sup\Bigl\{ \frac{ |g(x)|}{v(x)} : x\in\state\Bigr\};
\]
cf.\ 
	\cite{vei69,kar85b,kar85a}
	and the discussion in \cite{meyn-tweedie:book2}. 
The corresponding Banach space is denoted
$\Lv  \eqdef \{ g\colon\state\to\Re : 
\|g\|_{v} <\infty\}$,
and the induced operator norm on linear operators $ K\colon\Lv\to\Lv$ 
is,
\[
\lll K \lll_v \eqdef 
 \sup \Bigl\{ \frac{\| K h\|_{v}}{\|h\|_{v}} 
:  h\in L^{v}_\infty,\ \|h\|_{v}\neq 0\Bigr\}.
\]
An analogous weighted norm is defined for
signed measures $\nu$ on $(\state,\clB)$
via,
$$\|\nu\|_v:=\sup\Big\{\frac{|\nu(h)|}{\|h\|_v}
:  h\in L^{v}_\infty,\ \|h\|_{v}\neq 0\Bigr\}.
$$

The operator on $\Lv$ induced 
by the resolvent kernel $R$ will be shown to
satisfy $\lll R\lll_v<\infty$ under (DV3) (see \Proposition{DV4rlx}),  and it is known that $\lll P^t \lll_v$ is uniformly bounded in $t$ under this condition (see the proof of Theorem~6.1 of \cite{meytwe93b}).
\archive{math comment (7) Page 4: You need to prove that the Pt keeps Lv1 invariant. Otherwise all later arguments have problem one way or the other. (For instance, you implicitly assumed that Ptg is an element in Lv1 in (8)). I suspect that under (DV3), this might be true, but a proof is needed.}

All of the approximations in this paper are obtained 
with respect to $\lll\cdot\lll_v$. Our main results 
are all based on  \Theorem{Rapprox} below, which 
establishes conditions ensuring that the semigroup 
$\{P^t\}$ of the process $\bfPhi$ can 
be approximated (in this weighted operator norm)
by a semigroup written in terms of finite-rank
kernels. In particular, \Theorem{Rapprox}  states that the Donsker-Varadhan 
condition (DV3) holds if and only if the process
$\bfPhi$ can be approximated by an HMM in operator norm.


The approximating HMM is based on a generator that is a finite-rank
perturbation of the identity, of the form,
\begin{equation}
 \clE = \kappa \Bigl[- I+ \IND_{C_0}\otimes\nu_1
	+ \sum_{i,j=1}^N  r_{ij} \,   
	\IND_{C_i}\otimes\nu_j \Bigr]
 \label{Kgenerate}
\end{equation}
where $\{C_i:1\leq i\leq N\}$ is
a finite collection of disjoint, precompact sets,
$C_0$ is the complement 
of their union,
$\state\setminus\cup_{1\leq i\leq N} C_i$,
and $\{\nu_i\}$ are probability measures on $(\state,\clB)$
with each $\nu_i$ supported on $C_i$.
The constants  $\kappa$ and $\{r_{ij}\}$ are nonnegative, and the 
$\{r_{ij}\}$  define a transition matrix on   
the finite set $\{1,2,\ldots,N\}$.  
The approximating semigroup is expressed as the exponential family,
\begin{equation}
 Q^t = e^{t\clE},\qquad t\geq 0, 
 \label{KQ}
\end{equation}
where the exponential is defined via the usual 
power-series expansion. The family of resolvent kernels of the semigroup $\{Q^t\}$ 
is denoted $T_\alpha$, $\alpha>0$, where,
\begin{equation}
T_\alpha=\int_0^\infty e^{-\alpha t}  Q^t\, dt.
\label{KQresolve}
\end{equation}
The generator $\clE$ will be constructed so that $T_\alpha$ approximates $R_\alpha$ in $\Lv$ for $\alpha$ in a neighborhood of unity (see 
\Proposition{GeneratorResolveMain}).


While connections between separability and condition~(DV3) 
were previously established in \cite{wu04,kontoyiannis-meyn:II},   
\Theorem{Rapprox} goes well beyond prior work.   In particular, 
the equivalence between (DV3) and the finite-state HMM approximation in 
the strong sense given in the theorem
cannot be foreseen based on earlier results.
Although the main results 
of \cite{wu04,kontoyiannis-meyn:II} admit extensions to Markov models 
in continuous time,  essential properties of a diffusion must be exploited 
to obtain the uniform bound
\eqref{e:semiApprox}.

\begin{theorem}
\label{t:HMMapprox}
{\sc [(DV3) $\Leftrightarrow$ HMM approximation]}
For a Markov process $\bfPhi$ on $\state$ satisfying
conditions {\em (A1)}, {\em (A2)} and {\em (A3)}, 
the following are equivalent:

\begin{romannum}
\item {\sc Donsker-Varadhan Assumption:\ }   
Condition {\em (DV3)} holds 
in the form given in~{\em (A4)}.

\item {\sc HMM approximation:\ }
There exists a continuous function 
$v\colon\state\to [1,\infty)$ with compact sublevel sets
(possibly different from the 
function $v$ in {\em (i)}),
such that the following approximations hold:
For each $\epsy>0$ and $\delta\in (0,1)$, 
there exists a semigroup $\{Q^t\}$ as in \eqref{KQ} 
with generator $\clE$ of the form
given in \eqref{Kgenerate} and with
an associated family of resolvent kernels
$\{T_\alpha\}$ as in \eqref{KQresolve},
satisfying the following:
\begin{alphanum}
\item {\em Resolvent approximation:}  The resolvent kernels 
\eqref{resolvent} and  \eqref{KQresolve}  satisfy,
\[
\lll R_\alpha - T_\alpha \lll_{v}  \le \epsy,\qquad  
\delta\le \alpha \le \delta^{-1}.
\]
\item {\em Semigroup approximation:} 
\begin{equation}
\| P^t g - Q^t g \|_{v} \le 
\epsy( \|  g\|_v+ \| \clD^2 g\|_v) ,\qquad t\ge 0\,,
\label{e:semiApprox}
\end{equation}
for each $C^4$ function $g$ with compact support.


\item {\em Invariant measure approximation: }
The two semigroups have unique invariant probability measures $\pi$ 
and $\varpi$, satisfying, 
\archive{minor comment 6 addressed here by removing Psi}
\[
\|\pi -\varpi\|_{v} \le\epsy.
\]
\end{alphanum}
\end{romannum}
\end{theorem}

\begin{proof}
The proof is based on several results
contained in Section~\ref{finite}:

For the implication  (i) $\Rightarrow$ (ii), the function $v$ appearing in (A4) can be chosen the same as the function $v$ appearing in (iia)--(iic).
The implication (i) $\Rightarrow$ (iia) is contained 
in \Proposition{GeneratorResolveMain}; 
the implication (i) $\Rightarrow$ (iib) follows 
from \Proposition{HMMapproxPPkappa} combined with  
\Proposition{HMMapproxPPkappaB};
and the implication (i) $\Rightarrow$ (iic)
is given in Corollary~\ref{t:hidden}.

Finally, the implication (ii) $\Rightarrow$ (i) 
follows from Proposition~\ref{t:RapproxConverse}:  Under (ii) 
it follows that (A4) holds for continuous functions
$V_-, W_-$,  where $V_-\in L_\infty^{V}$.
\end{proof}


We next consider the probabilistic side of this theory,
and we show that a Markov process with generator of the form 
given in \eqref{Kgenerate} admits a 
representation as a finite-state
 hidden Markov model.

\subsection{Hidden Markov model approximations}
\label{HMMapprox}

A finite-state space hidden Markov model (HMM) in continuous time 
is defined as a pair $(\Upsilon,\bfmI)$,  where $\bfmI$ is a Markov 
process with finite state space $\state_I$.
The first component
$\Upsilon$ is called the observation process; it is 
a stochastic process taking values in some set $\ystate$.  
The joint dynamics are described as follows:  There is a family 
of probability measures $\{\nu_i\}$  on $\ystate$ such that,
for all measurable $A\subset \ystate$,
\[
\Prob\{ \Upsilon(t)\in A \mid( \Upsilon(s), I(s)) ,\ s <t ;\  I(t) = i\} = \nu_i(A) ,\quad i\in \state_I\, .
\]  
Here we explain how, under our conditions,
the continuous time Markov process $\bfPhi$ 
may be approximated by the finite-state process 
$\Upsilon$ of an appropriately constructed HMM. 
In fact, here the HMM will be special, 
in that the process $\Upsilon$ itself will be Markovian.   

Recall that the generator $\clD$ of $\bfPhi$ will be 
approximated by a generator $\clE$ of the form given 
in \eqref{Kgenerate}.
Let $\ystate$ denote the compact set 
$\ystate \eqdef\overline{\bigcup_{i=1}^N C_i}$,  and let 
 $\bfPsi$ denote the continuous-time Markov process
with generator $\clE$, and with corresponding transition
semigroup $\{Q^t\}$ defined in \eqref{KQ}.    The process $\bfPsi$ will define the observation process $\Upsilon$ in our HMM approximation.

A probabilistic description of $\bfPsi$  is based on a sequence of 
jump times $\{\tau_k : k\ge 0\}$, with $\tau_0:=0$. 
The description of $\tau_1$ depends on the  initial 
condition $\Psi(0)=x$:  Let $i$ denote the unique index for which
$x \in C_i$.   If $i\neq 0$, we construct $N$ independent exponential 
random variables with respective means equal to 
$\{(\kappa r_{ij})^{-1} : 1\le j\le N\}$, and the 
first jump after time $\tau_0:=0$ is defined as the minimum 
of these exponential random variables.  If $i=0$, i.e., $\Psi(0)=x\in C_0$,
 then   $\tau_1$  is 
given by the value of an exponential random variable
with mean $1/\kappa$.   Letting $j$ denote the index 
corresponding to the minimizing exponential 
random variable if $i\neq 0$, or taking $j=1$ if $x\in C_0$,
we define $\Psi(t) =x$ for $0=\tau_0\le t< \tau_1$, 
and let  $\Psi(\tau_1)$ be a sample from 
the distribution $\nu_j$.
\archive{minor comment 7}

This procedure is continued 
iteratively to define the sequence of sampling times 
$\{\tau_k\}$ along with the jump  process $\bfPsi$.
To see that $\bfPsi$ can be viewed as an HMM
we first present a simplified expression for the semigroup
$\{Q^t\}$. 

\begin{proposition}
\label{t:HMMa}
Consider the process $\bfPsi$ with generator
$\clE$ as in {\em (\ref{Kgenerate})} and semigroup 
$\{Q^t\}$ as in {\em \eqref{KQ}}. 
If the initial state $\Psi(0)$ is distributed
according to some probability measure
$\Psi(0)\sim \mu$ of the form $\mu=\sum_{i=1}^N p_i \nu_i$, 
where the vector $p=(p_1,p_2,\ldots,p_N)\in \Re^N_+$ 
satisfies $\sum p_i=1$, then the 
distribution $\mu Q^t$ of $\Psi(t)$ at time $t>0$ 
can be expressed as,
$$ \mu Q^t = \sum p_i(t) \nu_i,\;\;\;\;t>0,\;\;\;\;
\mbox{where}\;\;\;\;
p(t) = e^{-\kappa (I-r)t} p\, ,
$$
and $\kappa,r=\{r_{ij}\}$ are the coefficients
of the generator $\clE$ in \eqref{Kgenerate}.
\end{proposition}

\begin{proof}
It suffices to prove the result with $\mu=\nu_i$ for some $i$;
the general case follows by linearity. 

The power series representation of $Q^t$ 
implies that $\nu_i Q^t$ 
can be expressed as a convex combination of $\{\nu_j\}$ for each $t$,
\begin{equation} 
   \nu_i Q^t 
=   \sum_{j=1}^N    \varrho_{ij}(t) \nu_j\, ,\qquad t\geq 0.
\label{PK}
\end{equation}
An expression for the coefficients $\{\varrho_{ij}(t)\}$ can
be obtained from the  
differential equation,
\[
\ddt Q^t  = \clE Q^t,
\] 
as follows: Writing 
$\nu_i \clE = \kappa\bigl( - \nu_i + \sum_{j=1}^N r_{ij} \nu_j\bigr)$, 
we conclude that, for any $t\geq 0$,
\[
\ddt \nu_i Q^t 
=  \kappa \Bigl[- \nu_i + \sum_{j=1}^N   r_{ij}  \nu_j \Bigr] Q^t 
=  \kappa  \sum_{k=1}^N \Bigl[- \varrho_{ik}(t)\nu_k + \sum_{j=1}^N   r_{ij}     \varrho_{jk}(t)\nu_k \Bigr] .
\]
Therefore, the coefficients $ \varrho\in\Re^{N^2}$ 
appearing in \eqref{PK}  satisfy,
\[
\ddt \varrho_{ik} (t) 
=  \kappa \Bigl[- \varrho_{ik}(t)+ \sum_{j=1}^N r_{ij}\varrho_{jk}(t)\Bigr] .
\]
Given the initial condition $\varrho_{ij} (0)=I$, 
the solution to this ODE is given by,
$\varrho(t)=e^{-\kappa(I-r) t},$
$t\geq 0$, as required.
\end{proof}

For the HMM construction, let $\bfmI$ denote 
a finite-state, continuous-time Markov 
process, 
with values in $\{0,1,2,\dots,N\}$.
Its rate matrix is denoted by $q_{ij}\eqdef \kappa r_{ij} $ 
for $i\neq j$, and 
$q_{ii}\eqdef -\sum_{j\neq i} q_{ij}$.   
We take $r_{01}=1$ and $r_{0i}=0$
for all $i\neq 1$.

Written as an $(N+1)\times (N+1)$ 
matrix, this becomes $q= -\kappa(I-r)$.   
The process $\bfmI$ is the hidden state process;
the set $\state_I=\{1,2,\dots,N\}$ will be an absorbing set for $\bfmI$. 
Conditional on $\bfmI$, we define the  
observed HMM process, denoted $\bfUpsilon=\{\Upsilon(t)\}$, 
as follows.
Letting $\{\tau_i\}$ denote the successive 
jump times of $\bfmI$, 
$\Upsilon(t)$ is constant on the interval 
$t\in[\tau_i,\tau_{i+1})$, and satisfies
for each $ A\in\clB$ and $ i=0,1,2,\ldots$,
\[
\Prob\{\Upsilon(\tau_i)\in A \mid \Upsilon(t), 
t<\tau_i;  \ I(t), t< \tau_i;\ I(\tau_i)=k\} =
\Prob\{\Upsilon(\tau_i)\in A \mid  I(\tau_i)=k\} =\nu_k(A) .
\]
An immediate consequence of the definitions is that $\bfPsi$
can be expressed as an HMM:
\begin{proposition}
\label{t:HMMb}
Suppose that $\Psi(0)\sim \nu_i$ for some $i\geq 1$, and that 
the HMM is initialized in state $i$, i.e., $I(0)=i$.  
Then the jump process $\bfPsi$ and the HMM $\bfUpsilon$ 
are identical in law.
More generally, if $\Psi(0)=x\in C_i$ and $I(0)=i$, for
some $i=0,1,\ldots,N$, then the jump process $\bfPsi$ and 
the HMM $\bfUpsilon$ are identical in law
following the first jump,
\[
\{\Psi(t): t\ge \tau_1\} \eqdist  \{\Upsilon(t): t\ge \tau_1\}.
\]
\end{proposition}

\subsection{Separability and the spectrum}

The key property we will use to establish that 
a process $\bfPhi$ can be approximated by an 
HMM as in Theorem~\ref{t:HMMapprox}
will be the ``$v$-separability'' of its
resolvent $R$. Following \cite{kontoyiannis-meyn:II}
we say that a kernel $K$ is \textit{$v$-separable}
with respect to some function $v:\state\to[1,\infty)$,
if $\lll K  \lll_{v} <\infty$ and, for each $\epsy >0$, 
there exists a compact set $\ystate\subset\state$ 
and a finite-rank, probabilistic kernel $T$ 
supported on $\ystate$, such that 
$\lll K - T \lll_{v} \le \epsy$.   
By `finite-rank' we mean there are functions $\{ s_i\}$, 
measures $\{\nu_j\}$, and nonnegative constants $\{\theta_{ij}\}$ 
such that,
\begin{equation}
  T = \sum_{i,j=1}^N \theta_{ij} s_i\otimes\nu_j .
 \label{e:K}
\end{equation}
A kernel $T$ is `probabilistic' if $T(x,\state)=1$ for all $x\in\state$.

Our next result gives an alternative characterization
of the Donsker-Varadhan condition (DV3), showing that
it is equivalent to $v$-separability of the resolvent.
A similar result in  discrete time appears in \cite{wu04,kontoyiannis-meyn:II}. 
The implication   (ii) $\Rightarrow$ (i)   is contained 
in \Proposition{RapproxConverse}.   The forward implication (i) 
$\Rightarrow$ (ii) follows from \Proposition{i-ii}.   

\begin{theorem}
\label{t:Rapprox}
{\sc [(DV3) $\Leftrightarrow$ $v$-Separability]}
For a Markov process $\bfPhi$ on $\state$ satisfying
conditions {\em (A1)} and {\em (A2)},
the following are equivalent:
\begin{romannum}
\item {\sc Donsker-Varadhan Assumptions:\ }   
Condition {\em (DV3)} holds in the form
given in~{\em (A4)}.

\item {\sc $v$-Separability:\ }
The resolvent kernel $R$ is $v$-separable,
for a continuous function $v$ with compact
sublevel sets, possibly different from the
one in {\em (i)}.
\end{romannum}
\end{theorem}

The following result follows immediately 
from  \Theorem{Rapprox} and \Proposition{GeneratorResolveMain}, combined with \cite[Theorem 3.5]{kontoyiannis-meyn:II}.
Recall that the \textit{spectrum} $\clS(K)\subset \Co$ of a linear operator 
$K$ on $\Lv$  is the set of $z\in\Co$  such that the inverse 
$[Iz - K]^{-1}$  does not exist as a bounded linear operator 
on $ \Lv $.   
\begin{theorem}
\label{t:specDis}
{\sc [(DV3) $\Rightarrow$ Discrete Spectrum]}
Let $\bfPhi$ be a Markov process satisfying
conditions {\em (A1)} and {\em (A2)}. If
$\bfPhi$ also satisfies the
drift condition {\em (DV3)} in the form given
in {\em (A4)}, then the spectrum of the resolvent kernel 
is discrete in $\Lv$.
\end{theorem} 

\subsection{Extensions}
\label{s:extend}

Further connections between (DV3), $v$-separability,
multiplicative mean ergodic theorems, and large
deviations for continuous-time Markov processes
will be considered in subsequent work, generalizing
and extending the discrete-time results of
\cite{kontoyiannis-meyn:II}.
In particular, under (DV3), the process
$\bfPhi$ is ``multiplicatively regular'' and
satisfies strong versions of the
``multiplicative mean ergodic theorem.''
These results, in turn, can be used to deduce
a large deviations principle for the empirical
measures induced by $\bfPhi$. Moreover,
the rate function can be expressed in terms
of the entropy 
rate,  
as in 
\cite{donsker-varadhan:I-II,dembo-zeitouni:book,kontoyiannis-meyn:II}.

The technical arguments used
in the proofs of all the central results here
can easily be extended beyond the class of
continuous-sample-path diffusions in $\RL^d$.
Although such extensions will not be pursued
further in this paper, we note that the assumption
(A1) can be replaced by the condition that
$\bfPhi$ is a nonexplosive Borel right process
(so that it satisfies the strong
Markov property and has right-continuous
sample paths) on a Polish space $\state$.
Assumptions (A2) and (A3) can be maintained as stated;
the conclusions of  \Proposition{hypoB} continue to hold in 
this more general setting.
Assumption (A4) can also be maintained without modification.     
The resolvent equations  in \Proposition{ResolveDR} hold in this general 
setting,  
which is what is required in the converse theory 
that provides the implication (ii) $\Longrightarrow$ (i) 
in \Theorem{HMMapprox}.

Finally, there are applications to consider, as well as bridges to other areas such as statistics, machine learning, and operations research \cite{busvlisch12,denmehmey11}.
The approximation introduced in this paper is similar to the approximation performed in the modeling technique known as \textit{probabilistic latent semantic analysis} (PLSA); see \cite{hof01} for the basic concepts,  and \cite{gaugou05,sharajsma08} for surveys that describe connections with techniques from other fields.  Given a large $m\times m$ matrix $P$ representing associations between different objects,  the goal is to find an approximating matrix $T$,  an  $m\times r$ matrix $S$, and an $r\times m$ matrix $N$ such that $r\ll m$ and,  
\[
T = SN = \sum_{i=1}^r s_i n_i^\transpose,
\]
where $\{s_i:1\le i\le r\}$ denote the columns of $S$,  
and $\{n_i^ \transpose :1\le i\le r\}$ denote the rows of $N$.  Hence, 
the goal is to find a transition matrix of reduced rank, exactly as in 
this paper.   Our work provides motivation and rigorous justification
for the use PLSA models, even when the state space is general, and even 
for Markov models evolving in continuous time, as well as motivation for 
the development of approximation theory for diffusions based on observed 
trajectories of the process.

\bigskip

The remainder of the paper is organized as follows.   
The following section develops results establishing 
approximations between the process $\bfPhi$ and 
a simple jump process.  This is a foundation for Section~\ref{finite} 
that establishes similar approximations with an HMM. 


\section{Resolvents and Jump-Process Approximations} 
\label{ResolveApprox}

We begin in this section with an approximation of the 
process $\bfPhi$ by a pure jump-process denoted $\bfPhi^\kappa$,  
evolving on the 
state space $\state$.    The fixed constant  $\kappa>0$ denotes 
the jump rate.   The jump times $\{\tau_i : i\ge 0\}$ define a Poisson process:  
\archive{Reviewer doesn't know what undelayed means, and I don't want to define it}
$\tau_0=0$, and the increments are i.i.d.\ with exponential distribution and mean $\kappa^{-1}$.  
At the time of the $i$th jump we have $\Phi^\kappa(\tau_i)\sim \kappa R_\kappa(x,\varble)$, given 
that $\Phi^\kappa(\tau_{i-1}) =x$.     This process is Markov, 
with generator,
\be
\clD_\kappa \eqdef\kappa[- I +\kappa R_\kappa ].
\label{eq:Dkappa}
\ee
This is the generator for the Markov process used in the proof 
of the Hille-Yosida theorem in \cite{rogers-williams:I}.

Throughout this section it is assumed that $\lll R_\kappa\lll_v<\infty $, 
with $v$ being continuous, 
with compact sublevel sets.  Hence the generator $\clD_\kappa$ 
also has finite norm.   
This is justified by the following proposition, whose proof may be found in the Appendix.  The following drift condition is a relaxation of (DV3),
\begin{equation}
\clD v\le -v +b_v,
\label{DV4rlx}
\end{equation} 
where $b_v$ is a finite constant, and $v\colon\state\to[1,\infty)$.

\begin{proposition}
\label{t:DV4rlx}
Let $\bfPhi$ be a Markov process satisfying~{\em (A1)}.
\begin{romannum}
\item If   {\em (A4)} holds, then 
there is a function 
$v\colon\state\to[1,\infty)$ and 
a finite constant $b_v$ satisfying  \eqref{DV4rlx}.

\item
If \eqref{DV4rlx} holds for a function 
$v\colon\state\to[1,\infty)$ and 
a positive constant $b_v$, then
the following bounds hold,
\[
\begin{aligned}
\lll (\alpha R_\alpha)^n\lll_v & \le 1+b_v,\qquad 
	\hbox{\it for all}\;n\ge 1,\  \alpha >0;
\\
\pi(v)&\le b_v,\qquad 
\qquad \hbox{\it for any invariant probability measure $\pi$.}
\end{aligned}
\]
where $b_v$ is the constant in  \eqref{DV4rlx}.
\end{romannum}
\qed
\end{proposition}

We next review some background on $\psi$-irreducible Markov processes.  

\subsection{Densities, irreducibility and ergodicity}
\label{densities}


The density condition (A2) combined with the existence 
of a Lyapunov function as in (DV3) implies ergodicity.
Recall that a Markov process $\bfPhi$ with a unique invariant
probability measure $\pi$ is called $v$-uniformly
ergodic for some function $v:\state\to\RL$, if
there are constants $\beta_0>0$, $B_0<\infty$, such that,
\[
\lll P^t - \One\otimes\pi\lll_v \le e^{B_0 -\beta_0 t},\qquad t\geq 0.
\]
See \cite{meytwe93a} for basic theory of $\psi$-irreducible Markov processes, including definitions of small sets and aperiodicity in this general state-space setting.
\archive{minor comment (11):   Page 8, Proposition 2.2: give reference for definition of -irreducibility and aperiodicity. Also, smallness of a set is not defined or referenced up to this point, neither is
v-uniform ergodicity.
\\
I don't understand his last complaint about v-uni.  }
 
\begin{proposition}
\label{t:hypoB}
If conditions~{\em (A1), (A2)} and~{\em (A3)} hold, then   
the Markov process $\bfPhi$  
is $\psi$-irreducible and aperiodic 
with $\psi(\varble)\eqdef R (x_0,\varble)$,
and all compact sets are small.
If, in addition, {\em (DV3)} holds, 
then the process is $v$-uniformly ergodic 
with $v=e^V$.
\end{proposition}

\begin{proof}
Under (A1) and (A2) the Markov process is a 
T-process, since $R$ has the strong Feller property  \cite{meytwe93a}.  This combined with (A3) 
easily implies $\psi$-irreducibility   
with $\psi(\varble) = R (x_0,\varble)$.   Under (A3), for any 
set $A$ satisfying $\psi(A)>0$,  we have 
$P^t(x,A)>0$ for all $t\ge 0$ sufficiently large.   The proof is 
similar to the proof of Proposition~6.1 of  
\cite{meytwe93a}.    Hence the process is aperiodic.   
To see that all compact sets are small, we note that all 
compact sets are petite by \cite[Theorem~4.1]{meytwe93a}.   
Under aperiodicity,  petite sets are small; this is proved 
as in the discrete-time case \cite[Theorem 5.5.7]{meyn-tweedie:book2}.

To see that $\bfPhi$ is $v$-uniformly ergodic note that,
under (DV3), we have,  
\[
\clD v\le -\delta v +b_v^0\ind_C\, ,
\]
where $b_v^0= b \sup_{x\in C} v(x)$. 
This is condition (V4) of
\cite{dowmeytwe95a},
and hence the
conclusion follows from the main result of
\cite{dowmeytwe95a}.
\end{proof} 

Ergodic theory based on drift conditions such as (V4) is based 
in part on the following \textit{Comparison theorem};
see \cite{meyn-tweedie:book2} for the discrete-time counterpart.

\begin{proposition}
\label{t:CT}
If $\clD h \le - f + g$ for  nonnegative functions $(h,f,g)$,
and if $h$ is continuous, then
\begin{romannum}
\item 
For any $T>0$,
\begin{equation}
\Expect_x\Bigl[ h(\Phi(T)) + \int_0^T f(\Phi(t))\, dt\Bigr] 
\le h(x) + \Expect_x\Bigl[   \int_0^T g(\Phi(t))\, dt\Bigr].
\label{e:CT}
\end{equation}

\item
For any $\alpha>0$,
\[
\alpha R_\alpha h +  R_\alpha f  \le   h   +  R_\alpha g.
\]
\end{romannum}
\end{proposition}

\begin{proof} 
The proof of (i) is precisely the same as in the proof of the 
comparison theorem in discrete time 
\cite{meyn-tweedie:book2}.   Part (ii) follows from (i) on multiplying each side of \eqref{e:CT}
by $\alpha e^{-\alpha T}$, and integrating over $T\in\Re$.
\end{proof}

\subsection{Resolvent equations}

Recall the construction of the process $\bfPhi^\kappa$
with generator $\clD_\kappa$ as in (\ref{eq:Dkappa}).
We denote the semigroup of $\bfPhi^\kappa$ by
$P^t_\kappa := e^{t\clD_\kappa}$, $t\geq 0$, and 
its associated family of 
resolvent kernels by $R_{\kappa,\, \alpha}$:
\be
R_{\kappa,\alpha} \eqdef   \int_0^\infty  e^{-\alpha t} P_\kappa^t\, dt,
\;\;\;\;\alpha>0.
\label{eq:Rkappaalpha}
\ee 
\Proposition{ResolveDR} states the resolvent equations, 
and establishes some simple corollaries.   

\begin{proposition}
\label{t:ResolveDR}
Suppose the process $\bfPhi$
satisfies~{\em (A1)} and $\bfPhi^\kappa$
is the jump process with generator $\clD_\kappa$ as
in {\em (\ref{eq:Dkappa})}. 
Then, for any positive
constants $\alpha,\beta$ we have:
\begin{romannum}
\item 
The resolvent equation holds,
\begin{equation}
R_\alpha   =  R_\beta  +  (\beta-\alpha) R_\beta R_\alpha 
=  R_\beta  +  (\beta-\alpha)  R_\alpha R_\beta.
\label{REkun-5}
\end{equation}
\item 
For each $\alpha>0$ and any measurable function   
$h\colon\state\to\Re$ for which  $R_\alpha |h| $ is finite-valued,   
the function $f=R_\alpha h$ is in the domain of $\clD$,
and,
\begin{equation}
\clD R_\alpha h  = \alpha R_\alpha h   - h  .
\label{ResolveDR}
\end{equation}
Moreover, with $g= \alpha R_\alpha h   - h $ the stochastic process \eqref{e:extgenMart}
 is a martingale, so that \eqref{e:extgenMartA} holds.

\item
The resolvent of $\bfPhi^\kappa$ satisfies 
the analogous identity, 
\begin{equation}
\clD_\kappa R_{\kappa,\, \alpha} h  = \alpha R_{\kappa,\, \alpha} h   - h  ,
		 \qquad \text{if   $R_{\kappa,\, \alpha} |h|$ is finite valued.}
\label{RkappaGenerate}
\end{equation}

\item
The generators for $\bfPhi$ and $\bfPhi^\kappa$  are related by,
\begin{equation}
\clD_\kappa h = \clD [\kappa R_\kappa] h    \qquad \text{if $R_\kappa |  h|$ is finite valued};
\label{clDclDkappa}
\end{equation} 
\end{romannum}
\end{proposition}

\begin{proof}
Part (i) is the usual resolvent equation \cite{ethkur86}.
Part (iii) follow directly from (ii), 
and (iv) follows from (i) and (ii). 

It remains to prove the resolvent equation \eqref{ResolveDR} in the strong form:   \eqref{e:extgenMartA} holds with  $f=R_\alpha h$   and $g= \alpha R_\alpha h   - h $.   We have by Fubini's theorem,
\[
P^T f = \int_0^\infty e^{-\alpha t} P^{t+T}  h\, dt =  e^{\alpha T}\int_T^\infty e^{-\alpha t} P^{t}  h\, dt .
\]
Suppose first that $h$ is bounded.  It follows from Assumption~A2 then $P^t h$ is a continuous function of $t$.  
Hence $P^T f$ is $C^1$ with,
\[
\frac{d}{dT}P^T f  = \alpha e^{\alpha T}\int_T^\infty e^{-\alpha t} P^{t}  h\, dt  -  e^{\alpha T} P^T h
= P^T g . 
\]
The identity \eqref{e:extgenMartA} thus holds, by the fundamental theorem of calculus. 

If $h$ is not bounded we can construct a sequence of functions $\{h_n\}$ satisfying $|h_n(x)|\le \min(|h(x)|,n)$ for each $n$ and $x$,  and $h_n(x)\to h(x)$ as $n\to\infty $ for each $x$.   We then have for each $n$ and $t$,  with
$f_n=R_\alpha h_n$   and $g_n= \alpha R_\alpha h_n   - h_n $, 
\[
 P^t f_n     = f_n + \int_0^t P^s g_n\, ds  .
\]  
Under the assumption that $R_\alpha |h| $ is finite-valued,   
it follows that $ P^t |f|$ and $ \int_0^t P^s |g|\, ds$
are finite-valued.   The desired conclusion \eqref{e:extgenMartA}  
thus follows by dominated convergence. 
\end{proof}

The resolvent equation \eqref{ResolveDR} implies that $[\alpha I  - \clD]$ 
is a left inverse  of $R_\alpha$ for any $\alpha>0$, 
in the sense that $[\alpha I  - \clD]R_\alpha f =f$ 
for an appropriate class of functions $f$.
\archive{Reviewer wants us to be more precise.  I think this is precise enough - see the proposition. }
While $R_\alpha$ {\em cannot} be expressed
as a true operator inverse on the space
$L_\infty^v$, it is in fact possible to obtain 
such a representation for $R_{\kappa, \alpha}$.
This is made precise in the 
following.

\begin{lemma}
\label{t:Rkappa_alpha}
Suppose the process $\bfPhi$
satisfies {\em (A1)} and the drift
condition \eqref{DV4rlx}.
Then, for any $\alpha>0$,
 \begin{equation}
R_{\kappa,\, \alpha}   = 
[\alpha I - \clD_\kappa]^{-1} =  
[\alpha I - \kappa(\kappa R_\kappa -I)  ]^{-1} =  
 \frac{\kappa}{(\kappa+\alpha)^2} 
\sum_{n=-1}^\infty (1+\alpha\kappa^{-1})^{-n} ( \kappa R_\kappa)^{n+1},
\label{RRapproxB}
\end{equation}
where the sum converges in $\Lv$. Moreover,
\begin{equation}
\lll   \alpha R_{\kappa,\, \alpha}   \lll_v \le 1+b_v.
\label{RRapproxBbdd}
\end{equation}
\end{lemma}

\begin{proof} 
For any $n\ge 0$,   $\kappa>0$,  we have the bound 
$\lll ( \kappa R_\kappa)^{n+1} \lll_v  \le 1+b_v $, 
from \Proposition{DV4rlx}~(ii).
The representation \eqref{RkappaGenerate} implies that the inverse 
can be expressed as the power series 
(\ref{RRapproxB}), which 
is convergent in $\Lv$.
Since 
$\alpha I - \clD_\kappa$ is a left inverse of 
$R_{\kappa,\, \alpha}$,  it then follows that  $R_{\kappa,\, \alpha}   = 
[\alpha I - \clD_\kappa]^{-1} $.
 
To establish the bound \eqref{RRapproxBbdd} we apply the triangle inequality,
\[
\lll R_{\kappa,\, \alpha}   \lll_v  \le 
 \frac{\kappa}{(\kappa+\alpha)^2} \sum_{n=-1}^\infty 
(1+\alpha\kappa^{-1})^{-n} \lll \kappa R_\kappa\lll_v^{n+1} .
\]
Using once more the bound $ \lll \kappa R_\kappa\lll_v^{n+1}  
\le 1+b_v $, 
and simplifying the expression for the 
sum in the following bound,
\[
\lll R_{\kappa,\, \alpha}   \lll_v  \le 
 (1+b_v)  
 \frac{\kappa}{(\kappa+\alpha)^2}
\left( \frac{1+\alpha\kappa^{-1} }
{ 1-  (1+\alpha\kappa^{-1})^{-1}  }    \right),
 \]
we obtain the bound in \eqref{RRapproxBbdd}, as claimed.
\end{proof}

\subsection{Resolvent approximations}

Under (DV3) or, more generally, under the weaker
drift condition \eqref{DV4rlx},
we obtain the following strong approximation 
for the resolvent kernels:

\begin{proposition}
\label{t:RRapprox}
Suppose the process $\bfPhi$ 
satisfies~{\em (A1)} and $\bfPhi^\kappa$
is the jump process with generator $\clD_\kappa$ 
defined in {\em (\ref{eq:Dkappa})}. 
If $\bfPhi$ satisfies the drift condition
\eqref{DV4rlx},
then, for each $\alpha < \kappa$:
\begin{equation*}
\lll R_{\kappa,\, \alpha} -R_\alpha\lll_v \le  
\frac{4}{\kappa} (1+b_v) .
\end{equation*}
\end{proposition}

\begin{proof}
We first obtain a power series representation for $R_\alpha   $ not in 
terms of its generator, but in terms of the resolvent kernel $  R_\kappa  $.
The resolvent equation \eqref{REkun-5} with $\beta=\kappa$ 
and $\alpha>0$ arbitrary  
gives $ [I-  (\kappa-\alpha) R_\kappa] R_\alpha   =  R_\kappa  $.
Since $0<\alpha < \kappa$ and $R_\kappa(x,\state)=\kappa^{-1}$ for 
each $x$,  it follows that $R_\alpha$ can be expressed as the power series,
\begin{equation*}
R_\alpha   =  [I-  (1-\alpha \kappa^{-1}) \kappa R_\kappa] ^{-1}R_\kappa  
		=  \frac{1}{\kappa} \sum_{n=0}^\infty 
		(1-\alpha\kappa^{-1})^n (\kappa R_\kappa)^{n+1}.
\end{equation*}
\Proposition{DV4rlx}~(ii) gives the uniform bound,
$\lll ( \kappa R_\kappa)^{n+1} \lll_v  \le 1+b_v $,  
which implies that this sum converges in $\Lv$.

Applying \Lemma{Rkappa_alpha}, we conclude that the 
difference of the two resolvent kernels 
$ R_{\kappa,\, \alpha} $ and $R_\alpha$
can be decomposed into three terms:
\begin{subequations}
\begin{eqnarray}
R_{\kappa,\, \alpha} -R_\alpha
&=&    
\Bigl( \frac{\kappa}{(\kappa+\alpha)^2} -\frac{1}{\kappa} \Bigr)
\sum_{n=0}^\infty (1+\alpha\kappa^{-1})^{-n} ( \kappa R_\kappa)^{n+1}
\label{RbkapRbA}
\\ 
&& { } +  \frac{1}{\kappa} \sum_{n=0}^\infty \Bigl( (1+\alpha\kappa^{-1})^{-n}-(1-\alpha\kappa^{-1})^n\Bigr) (\kappa R_\kappa)^{n+1}
\label{RbkapRbB}
\\ 
&& { } + 
\Bigl(
 \frac{\kappa}{(\kappa+\alpha)^2}    
(1+\alpha\kappa^{-1})^{-n}\Big|_{n=-1} \Bigr) I.
\label{RbkapRbC}
\end{eqnarray}
\end{subequations}
To bound the first term \eqref{RbkapRbA} we apply \Proposition{DV4rlx}~(ii):
\[
\begin{aligned}
\frac{1}{(1+b_v)}
\lll \hbox{RHS of \eqref{RbkapRbA}} \lll_v 
& \le 
	 \Bigl| \frac{\kappa}{(\kappa+\alpha)^2} 
	-\frac{1}{\kappa} \Bigr|
	\sum_{n=0}^\infty (1+\alpha\kappa^{-1})^{-n}
	\\
& =   
	\Bigl( \frac{1}{\kappa}-\frac{\kappa}{(\kappa+\alpha)^2}  \Bigr)
	\Bigl(1- (1+\alpha\kappa^{-1})^{-1} \Bigr)^{-1}
	\\
& =  \Bigl( \frac{2\kappa\alpha+\alpha^2}
	{\kappa (\kappa+\alpha)^2} \Bigr)
	\Bigl(  \frac{\alpha}{\kappa+\alpha} \Bigr)^{-1}
	\\
& =  \frac{\kappa +(\kappa   +\alpha)}{ \kappa (\kappa+\alpha)}  =  \frac{1}{  \kappa+\alpha}  +  \frac{1}{  \kappa}    
	\, .
\end{aligned}
\]
This implies the bound,
\[
\lll \hbox{RHS of \eqref{RbkapRbA}} \lll_v \le\frac{2}{\kappa} (1+b_v) .
\]

The next inequality also uses the bound 
$\lll (\kappa R_\kappa)^n\lll_v\le 1+b_v$: 
\[
\begin{aligned}
\lll \hbox{RHS of \eqref{RbkapRbB}} \lll_v & \le 
 (1+b_v) \frac{1}{\kappa} \sum_{n=0}^\infty \Bigl( (1+\alpha\kappa^{-1})^{-n}-(1-\alpha\kappa^{-1})^n\Bigr)  
\\
& =
 (1+b_v) \frac{1}{\kappa}  
	\Bigl([1- (1+\alpha\kappa^{-1})^{-1}]^{-1}-
	[1-(1-\alpha\kappa^{-1})]^{-1}\Bigr)  \\
& =
 \frac{1}{\kappa}  (1+b_v).
\end{aligned}
\]
The final term \eqref{RbkapRbC} is elementary: 
\[
\lll \hbox{RHS of \eqref{RbkapRbC}} \lll_v = 
 \frac{\kappa}{(\kappa+\alpha)^2}   
 (1+\alpha\kappa^{-1}) =  \frac{1}{\kappa+\alpha} .
\]
Substituting these three bounds completes the proof.
\end{proof}
 

\section{Separability} 
\label{finite}

In this section we develop consequences of the separability assumption.   
In particular, we describe the construction of an approximating 
semigroup $\{Q^t\}$ with generator of the form
given in (\ref{Kgenerate}),
as described in \Theorem{HMMapprox}.
This is accomplished in four steps:

\begin{romannum}
\item 
First we note that under
(DV3) the resolvent kernel $R$ of $\bfPhi$
can be truncated to a compact set.

\item 
Then we argue that, again on a compact set,
$R$ can be approximated by a 
finite-rank kernel $T$.

\item
We next prove that the generator,
$\clD_\kappa \eqdef\kappa[- I +\kappa R_\kappa ]$,
of the jump process $\bfPhi^\kappa$ constructed 
in Section~\ref{ResolveApprox}, can be approximated
by a generator $\clE$ of the form 
\eqref{Kgenerate},
\ben
 \clE = \kappa \Bigl[- I+
	 \IND_{C_0}\otimes\nu_1
	+ \sum_{i,j=1}^N  r_{ij} \,   
	\IND_{C_i}\otimes\nu_j \Bigr],
\een
as long as $\kappa>0$ is chosen 
sufficiently large. This key result
is described in Proposition~\ref{t:i-ii}.

\item
Finally we show that the transition semigroup $\{P^t\}$
of the original process $\bfPhi$ can be 
approximated by the semigroup $\{P^t_\kappa\}$
of the jump process $\bfPhi^\kappa$
(Proposition~\ref{t:HMMapproxPPkappa}),
and that the semigroup $\{P^t_\kappa\}$
can in turn be approximated by the semigroup
$\{Q^t\}$ corresponding to an HMM with
a generator $\clE$ as above
(Proposition~\ref{t:HMMapproxPPkappaB}).

\end{romannum}

Again, the starting point of these results is justified by applying (DV3) to obtain the truncation described in (i).  A converse is obtained in the following result.  The proof is based on the resolvent equations, and 
is found in the Appendix.

\begin{proposition}
\label{t:RapproxConverse}
Suppose that the Markov process
$\bfPhi$ satisfies conditions~{\em (A1)} and~{\em (A2)},
and that its resolvent kernel $R$ is $v$-separable 
for some continuous function $v\colon\state\to[1,\infty)$
with compact sublevel sets.
Then~{\em (A4)} holds for some continuous
$V_-, W_-$ on $\state$, and the function
$V_-$ is in $L_\infty^{V}$.
\qed
\end{proposition}

\subsection{Truncations and finite approximations}

Let $\bfPhi$ be a Markov process satisfying
condition~(A1), with generator $\clD$
and associated resolvent kernels $\{R_\alpha\}$.
Recall the definition of the corresponding 
jump process $\bfPhi^\kappa$ in the beginning 
of Section~\ref{ResolveApprox}, 
with generator $\clD_\kappa$ and associated
resolvents $\{R_{\kappa,\alpha}\}$.

Our result here shows that 
condition (DV3) implies that the 
generator $\clD_\kappa$ of the jump
process $\bfPhi^\kappa$
can be approximated by a generator
$\clE$ as in (\ref{Kgenerate}).
This result is a corollary of \Proposition{finiteRank}, whose proof is given in the Appendix.
\archive{minor comment (16) Page 14, Proposition 3.1: There is no mentioning where a proof can be found. If you
want to keep a proof in Appendix, please say that the proof of Proposition 3.1 is
delayed until Appendix A.}
\begin{proposition}
\label{t:i-ii}
Suppose the Markov process $\bfPhi$ satisfies
conditions~{\em (A1)}, {\em (A2)}. If~{\em (DV3)} holds
as in assumption~{\em (A4)}, 
then, for each $\kappa>0$ and any $\epsilon>0$,
there exists a generator $\clE$ of the form given 
in \eqref{Kgenerate}, such that all the $r_{ij}$
are strictly positive,
and the generator
$\clD_\kappa$ of the jump process $\bfPhi^\kappa$
can be approximated in operator norm as,
\be
\lll     \clD_\kappa  -  \clE  \lll_v \le \epsy,\qquad  
\label{eq:nonUnif}
\ee
with $v:=e^{V}$.
\qed
\end{proposition}

From \Proposition{i-ii} we have a generator
$\clE$ of the form (\ref{Kgenerate}),
and with $Q^t:=e^{t\clE}$, $t>0$, being
the associated transition semigroup,
the corresponding resolvent kernels 
$\{T_\alpha\}$ are defined, as usual,
in (\ref{KQresolve}).
Using the approximation 
of the generator $\clE$
in (\ref{eq:nonUnif}), 
we next show that the 
kernels $\{T_\alpha\}$ 
can be expressed as operator
inverses, in a way analogous 
to the representations obtained
in \Lemma{Rkappa_alpha}
for the resolvents $\{R_{\kappa,\alpha}\}$.

\begin{lemma}
\label{t:newLemma}
Suppose that the assumptions of \Proposition{i-ii}
hold, and choose $\kappa>0$ and $\epsilon_0>0$
such that $\lll\clD_\kappa-\clE\lll_v\leq \epsilon_0$.
Then the resolvent obtained from the semigroup $\{Q^t\}$ 
can be expressed as an inverse operator on $\Lv$: 
For all $\alpha>(1+b_v)\epsilon_0$,
\begin{equation*}
T_\alpha = [\alpha I  -\clE]^{-1},
\end{equation*}
where $b_v$ is as in \Proposition{DV4rlx}~(i).
Moreover, for all such $\alpha$ we 
have the norm bound,
\begin{equation}
\lll T_\alpha \lll_v \le  \frac{1+b_v}{\alpha - (1+b_v) \epsy_0}. 
 \label{e:EkappaInvBdd}
\end{equation}
\end{lemma}

\begin{proof}
Note that we already have from the resolvent equation 
the formula $ [\alpha I  -\clE] T_\alpha = I$ on $\Lv$.   
It remains to show that  $ [\alpha I  -\clE] $ admits an inverse.  
We can write, on some domain,
\[
[\alpha I  -\clE]^{-1} = 
[\alpha I  - \clD_\kappa +\clD_\kappa  - \clE]^{-1} 
= 
R_{\kappa,\, \alpha} [ I + ( \clD_\kappa  - \clE) R_{\kappa,\, \alpha}  ]^{-1}. 
\]
The right-hand-side admits a power series representation whenever 
$\lll ( \clD_\kappa  - \clE) R_{\kappa,\, \alpha}  \lll_v<1$.
In fact, under the assumptions of the Lemma,
using the bound in \Lemma{Rkappa_alpha} we have, 
$$
\lll ( \clD_\kappa  - \clE) R_{\kappa,\, \alpha}  \lll_v
\leq
\lll \clD_\kappa  - \clE \lll_v\cdot
\lll R_{\kappa,\, \alpha}  \lll_v
\leq \epsilon_0(1+b_v)/\alpha<1,$$
and 
the resulting bound is precisely  \eqref{e:EkappaInvBdd}.
\end{proof}

Our next result shows that $v$-separability implies 
that each of the resolvent kernels $R_\alpha$
can be approximated by the kernels
$\{T_\alpha\}$ obtained from a finite-rank semigroup.
Specifically, $R_\alpha$ will be approximated
by a resolvent $T_\alpha$ of the form
\eqref{KQresolve}, where the transition
semigroup $\{Q^t\}$ is that of a Markov
process with generator $\clE$ as in (\ref{Kgenerate}).

\begin{proposition}
\label{t:GeneratorResolveMain}
Under the assumptions of \Proposition{i-ii},
for each $\epsilon>0$ and $\delta\in(0,1)$,
there exists a generator $\clE$ of the form 
given in \eqref{Kgenerate}, such that the 
corresponding resolvent kernels $\{T_\alpha\}$ 
defined in \eqref{KQresolve} satisfy the
following uniform bound:
\ben
\lll R_\alpha-T_\alpha\lll_v\leq\epsilon,
\;\;\;\;\;\;\delta\leq\alpha\leq\delta^{-1}.
\een
\end{proposition}

\begin{proof}
To establish the uniform bound 
in operator norm, first we approximate $R_\alpha$ 
by $R_{\kappa,\, \alpha}$.
Under (DV3), \Proposition{DV4rlx}~(i)
implies that we can use \Proposition{RRapprox}
as follows:  We fix $\kappa\ge\delta^{-1}$  such that the 
right-hand-side of this bound is no greater than $\half \epsy$, giving,
\begin{equation}
\lll   R_{\kappa,\, \alpha} -   R_\alpha \lll_v 
\le \half \epsy,\qquad \alpha\le \delta^{-1}. 
\label{GeneratorResolveMainA}
\end{equation}

We now invoke \Proposition{i-ii}:
Fix an operator $\clE $ of the form  \eqref{Kgenerate} satisfying,
\begin{equation*}
\lll \clD_\kappa - \clE \lll_v \le \epsy_0 ,
\end{equation*}
where $\epsy_0\in (0,\epsy)$ is to be determined.      
\Lemma{Rkappa_alpha} and \Lemma{newLemma} give,
\[
T_\alpha=  [\alpha I  -\clE]^{-1},
\quad 
R_{\kappa,\, \alpha}   = [\alpha I  - \clD_\kappa]^{-1} .
\] 
Hence the difference can be expressed,
\[
\begin{aligned}
   T_\alpha -   R_{\kappa,\, \alpha} &=  T_\alpha 
	[\clE -\clD_\kappa] R_{\kappa,\, \alpha}   \\
    &= [T_\alpha -   R_{\kappa,\, \alpha} ]   
	[\clE -\clD_\kappa]R_{\kappa,\, \alpha}   
	+ R_{\kappa,\, \alpha}[\clE -\clD_\kappa]R_{\kappa,\, \alpha} ,
\end{aligned}
\]
and applying the triangle inequality together with the 
sub-multiplicativity of the operator norm,
\[ 
  \lll T_\alpha -   R_{\kappa,\, \alpha} \lll_v 
   \le  
 \lll T_\alpha -   R_{\kappa,\, \alpha} \lll_v 
   \lll \clE -\clD_\kappa\lll_v\lll R_{\kappa,\, \alpha} \lll_v 
   + 
   \lll R_{\kappa,\, \alpha}\lll_v^2 \lll\clE -\clD_\kappa  \lll_v.
\]
\Lemma{Rkappa_alpha} gives the bound 
$\lll \alpha R_{\kappa,\, \alpha} \lll_v\le (1+b_v)$, 
and hence for $\alpha \in [\delta ,\delta^{-1}]$,
\[
\lll \clE -\clD_\kappa\lll_v\lll R_{\kappa,\, \alpha} \lll_v  
\le  \epsy_0 (1+b_v)/\delta .
\]
Assuming that $\epsy_0 >0$  is chosen so that  the right-hand-side 
is less than one, we can substitute into the previous bound 
and rearrange terms to obtain,
\[ 
  \lll T_\alpha -   R_{\kappa,\, \alpha} \lll_v  
  \le  \frac{  \lll R_{\kappa,\, \alpha}\lll_v^2 }{1-\lll \clE -\clD_\kappa\lll_v\lll R_{\kappa,\, \alpha} \lll_v} \lll \clE -\clD_\kappa  \lll_v
      \le \Bigl(\frac{ (1+b_v)^2 }{1-   \epsy_0 (1+b_v)/\delta}\Bigr) \Bigl(\frac{ \epsy_0  }{\delta^2 }\Bigr).
 \] 
Choosing $\epsy_0 =\fourth (1+b_v)^{-2} \epsy \delta^2 $ then gives,
\[
  \lll   T_\alpha -     R_{\kappa,\, \alpha} \lll_v  \le
   \fourth\epsy \frac{ 1 }{(1-\fourth \epsy)} 
    \le\half\epsy,\qquad  \alpha\in[\delta ,\delta^{-1}].
\]
This combined with \eqref{GeneratorResolveMainA}  
and the triangle inequality completes the proof. 
\end{proof} 

\subsection{Ergodicity}
\label{spec:linear}
 
To establish solidarity over an infinite time horizon we impose the 
reachability condition (A3) throughout the remainder of this section. 
Recall the construction of the approximating HMM process
$\bfPsi$ in Section~\ref{HMMapprox}, and the definition
of $v$-uniform ergodicity from Section~\ref{densities}.

\begin{proposition}
\label{t:HMMergodic}
Suppose the process $\bfPhi$ satisfies
conditions~{\em (A1) -- (A4)},
so that, in particular, $\bfPhi$ is 
$v$-uniformly ergodic with $v=e^V$
by \Proposition{hypoB}.
Then:
\begin{romannum}
\item
For each $\kappa>1$, 
the jump process $\bfPhi^\kappa$ 
is $v$-uniformly ergodic, with $v=e^V$.
\item
The HMM process $\bfPsi$ 
is $v$-uniformly ergodic, 
with $v=e^V$.
\end{romannum}
\end{proposition}

Before proceeding with the proof we prove Lyapunov bounds that are useful in later results.


\begin{lemma}
\label{t:V4approxModels}
Under the assumptions of \Proposition{HMMergodic}, there exist $\delta_\circ>0$ and $b_\circ<\infty$ such that the following bound holds for each  $\kappa>1$:
\begin{equation}
\clD_\kappa v\le -\delta_\circ v + b_\circ.
\label{e:JumpPhiKappaV4}
\end{equation}
Consequently, the following bound holds for the semigroup,
\begin{equation} 
\lll P_\kappa^t\lll_v \le 1+ b_\circ/\delta_\circ,\qquad t\ge 0.
\label{e:JumpPhiKappaV4cor}
\end{equation}
\end{lemma}

\begin{proof}
The bound \eqref{e:JumpPhiKappaV4cor}
 follows from
\eqref{e:JumpPhiKappaV4} using a version of the comparison 
theorem  (see eqn.~(31) of  \cite{dowmeytwe95a}):
\[
 P^t_\kappa v \le  e^{-\delta_\circ t} v + b_\circ/\delta_\circ.
\]

The proof of \eqref{e:JumpPhiKappaV4} begins with the bound $
\clD v\le  [-\delta   + \ind_C]v$,
which holds under (DV3) because $W\ge 1$ everywhere.   
Letting $b_v^0 = b\max_C v$ then gives,
\[
\clD v\le  -\delta v  + b_v^0 \ind_C \, .
\]
Applying \Proposition{CT} (ii) with $h=v$, $f=\delta v$ and $g\equiv  b_v^0 \ind_C$  implies that,
\[
\kappa R_\kappa v + \delta R_\kappa v  \le   v   + R_\kappa  g \le   v   + \kappa^{-1} b_v^0.
\]
On rearranging terms this gives,
\[
\kappa R_\kappa v  \le (1+\delta \kappa^{-1} ) ^{-1}(v+   \kappa^{-1}b_v^0  ),
\]
and thence,
\[
\kappa R_\kappa v  - v  \le  - \frac{\delta}{\delta+\kappa} v
+ \frac{1}{\delta+\kappa} b_v^0.
\]
From the definition of the generator for the jump process we 
conclude that the desired bound holds,
\[
\clD_\kappa v
=
\kappa[
\kappa R_\kappa v  - v ] \le  - \frac{\delta\kappa}{\delta+\kappa} v
+ \frac{\kappa}{\delta+\kappa} b_v^0.
\]
This gives 
\eqref{e:JumpPhiKappaV4} on choosing the worst-case over $\kappa\ge 1$: 
\[
\delta_\circ =\delta  /(\delta+1),\quad   b_\circ=   b_v^0 .
\]
\end{proof}

\begin{proof}[Proof of \Proposition{HMMergodic}]
To establish (i) we first demonstrate that  $\bfPhi^\kappa$ 
is  irreducible and aperiodic.
If $\psi$ is a maximal irreducibility measure for $\bfPhi$,  
then  \Lemma{Rkappa_alpha} implies that $\psi \prec R_\kappa(x,\varble)$ 
for each $x$.  This implies that the chain with transition kernel 
$\kappa  R_\kappa$ is $\psi$-irreducible and aperiodic.   
Irreducibility  and aperiodicity for $\bfPhi^\kappa$ is then 
obvious since it is 
a jump process with Poisson jumps, and jump 
distribution $\kappa  R_\kappa$.

To complete the proof of (i) we establish condition (V4) 
of  \cite{dowmeytwe95a}.   From \Lemma{V4approxModels} we obtain,
\[
\clD_\kappa v\le -\half \delta_\circ v +  b_\circ\ind_{C_\circ},
\]
where $C_\circ =\{ x :\half \delta_\circ v (x) \le   b_\circ\}$.    
The sublevel set $C_\circ$ is compact,  and   
Proposition~\ref{t:hypoB} implies that compact sets are small, so this implies 
that the jump process is $v$-uniformly ergodic.

Analogous arguments for  $\bfPsi$ will establish (ii):
$\psi$-irreducibility and aperiodicity are immediate by
\Proposition{HMMb} and the fact that
all $r_{ij}$ in the definition of $\clE$
are strictly positive, from \Proposition{i-ii}.
To show that $\bfPsi$ is $v$-uniformly ergodic
simply note that, by the definition of $\clE$,
$$\clE v =-\kappa v+\kappa Tv
\leq -\kappa v+\kappa b\IND_\ystate,$$
where $b:=\sup_{b\in\ystate}v(x)$.
Again, this is a version of
condition (V4) of \cite{dowmeytwe95a}, 
and the conclusion follows from 
\cite{dowmeytwe95a}.
\end{proof}  

\subsection{Semigroup approximations}
\label{semiapprox}
 
We begin with an approximation
bound between the semigroups corresponding 
to $\bfPhi$ and $\bfPhi^\kappa$.

\begin{proposition}
\label{t:HMMapproxPPkappa}
Suppose that $\bfPhi$ satisfies conditions~{\em (A1) -- (A4)}. 
Then there exists $b_\bullet$ depending only on $\bfPhi$ such that, 
\[
\| P^t g - P^t_\kappa g\|_v \le b_\bullet\kappa^{-1} \|\clD^2 g\|_v,\qquad t\ge 0,\ \kappa\ge 1,
\]
for any $C^4$ function $g$ with compact support.
\end{proposition}

\begin{proof}
Under the assumption of the proposition,   the local-martingale assumption can be strengthened to the martingale property \eqref{e:extgenMartA}.    
That is,  for any $T>0$,
\[
\begin{aligned}
\Expect_x\bigl[ g(\Phi(T))  \bigr]
&= 
g(x) + \Expect_x\Bigl[   \int_0^T \clD g (\Phi(t))\, dt\Bigr]
\\
\Expect_x\bigl[ \clD g(\Phi(T))  \bigr]
&= 
\clD g(x) + \Expect_x\Bigl[   \int_0^T \clD^2 g (\Phi(t))\, dt\Bigr]\,,
\qquad x\in\state.
\end{aligned}
\]
It follows that $P^T g$ is differentiable in $T$, and the same is true for $P_\kappa^Tg$.

Denote the 
difference $\epsy_g(t) = P^t g - P_\kappa^t g$.   We have for any $t$,
\[
\begin{aligned}
\ddt \epsy_g(t) & = P^t \clD  g - P_\kappa^t \clD_\kappa g
\\
& =  \clD_\kappa[P^t g -  P_\kappa^t g]  +   P^t [\clD  -\clD_\kappa] g
\\
& =  \clD_\kappa[P^t g -  P_\kappa^t g]  - \kappa^{-1}  P^t [\clD\clD_\kappa g],
\end{aligned}
\]
where in the second equation we have used here the fact that the 
operators $P^t$, $P_\kappa^t$, and $\clD_\kappa$ all commute.   
The final equation follows from \eqref{clDclDkappa} and the definition of $\clD_\kappa$ in \eqref{eq:Dkappa}.
Writing $h=\clD\clD_\kappa g$,  this can be solved to give,
\[
\epsy_g(t)=\epsy_g(0)  -  \kappa^{-1} \int_0^t e^{s\clD_\kappa} 
 P^{t-s} h\, ds.
\]
Substituting $P_\kappa^s = e^{s\clD_\kappa} $ and  $\epsy_g(0) =0$ simplifies this expression:
\begin{equation}
\epsy_g(t) =  - \kappa^{-1}  \int_0^t  P_\kappa^s   P^{t-s} h\, ds.
\label{HMMapproxPPkappaB}
\end{equation}

We have $\pi(h)=\pi( P^{t-s} h)=0$, so that by \Proposition{hypoB} we have for some $B_0<\infty $ and $\beta_0>0$,  
\[
\|P^{t-s} h\|_v\le  e^{B_0 -\beta_0 (t-s)}  \| h\|_v ,\qquad 0\le s\le t
\]
Consequently, for each $x\in\state$ and $t\ge 0$,  
\[
\begin{aligned}
\Bigl|  \int_0^t  P_\kappa^s   P^{t-s} h \, (x) \, ds \Bigr|
&\le 
 \| h\|_v e^{B_0}
 \int_0^t  e^{ -\beta_0 (t-s)}  P_\kappa^s   v \, (x) \, ds
\end{aligned}
\]
Recalling the bound \eqref{e:JumpPhiKappaV4cor} on $\lll  P_\kappa^s   \lll_v$ and substituting into \eqref{HMMapproxPPkappaB} 
gives $\|\epsy_g(t)\|_v \le \kappa^{-1} \beta_0^{-1} \| h\|_v e^{B_0} $.    We have $h=\clD\clD_\kappa g$, 
and hence  the generator relationship  \eqref{clDclDkappa} and  
the generator bound in \Proposition{DV4rlx}~(ii) give,
\[
\|h\|_v\le \lll \kappa R_\kappa\lll_v\|\clD^2 g\|_v   
\le(1+b_v)\|\clD^2 g\|_v.
\]
Finally, substituting this into the previous 
bound on $\|\epsy_g(t)\|_v $ completes the proof.
\end{proof}

Similar arguments provide approximation bounds  for the semigroups 
corresponding to $\bfPhi^\kappa $ and $\bfPsi$, where the latter 
is denoted $\{Q^t\}$ and defined in  \eqref{KQ}.

\begin{proposition}
\label{t:HMMapproxPPkappaB}
Suppose that $\bfPhi$ satisfies conditions~{\em (A1)--(A4)}. 
Then there 
exists $b_\bullet$ depending only on $\bfPhi$ such that for $g\in\Lv$,
\[
\| P^t_\kappa g-  Q^t  g\|_v \le b_\bullet \epsy \|  g\|_v.
\]
\end{proposition}

\begin{proof} 
The proof is similar to the proof of \Proposition{HMMapproxPPkappa}:
We fix $g\in\Lv$, and denote the error by,
\[
\epsy_g(t) =
P^t_\kappa g- Q^t g  ,\qquad t\ge 0.
\]
The right hand side is differentiable by construction of the two semi-groups, with 
\[ 
\ddt \epsy_g(t)  =  \clD_\kappa\epsy_g(t)  +  [\clD_\kappa-\clE ] Q^t g
\]
This can be solved to give,
\[
\epsy_g(t)  = 
\epsy_g(0)  + \int_0^t P_\kappa^s  [\clD_\kappa-\clE ] Q^{t-s} g\, ds.
\]
We have $\epsy_g(0) =0$.  Moreover, 
$ [\clD_\kappa-\clE ] 1\equiv 0$, which implies 
that $ [\clD_\kappa-\clE ] g =  [\clD_\kappa-\clE ][g-\varpi(g)]$.  
Here, $\varpi$ denotes the unique invariant measure of
the process $\bfPsi$, guaranteed to exist by Proposition~\ref{t:HMMergodic}.
Hence,
\[
\|\epsy_g(t)\|_v  
\le 
\lll \clD_\kappa-\clE\lll_v \int_0^t \lll P_\kappa^s \lll_v
 \| (Q^{t-s} -1\otimes\varpi)g\|_v\, ds.
\]

Substituting the bound $ \lll P_\kappa^s \lll_v\le 1+ b_\circ/\delta_\circ$ from 
\Lemma{V4approxModels} gives,
\[
\|\epsy_g(t)\|_v  
\le 
\lll \clD_\kappa-\clE\lll_v  (1+ b_\circ/\delta_\circ) \| g\|_v 
\int_0^t \lll Q^{t-s} -1\otimes\varpi\lll_v \, ds\,,
\]
and \Proposition{i-ii} gives $ \lll \clD_\kappa-\clE\lll_v \le \epsy$.  This establishes the result with
\[
b_\bullet =   (1+ b_\circ/\delta_\circ)  \int_0^\infty \lll  Q^r -1\otimes\varpi\lll_v \, dr \, ,
\]
which is finite, by  \Proposition{HMMergodic}.
\end{proof}

The following bound is an immediate consequence 
of the last Proposition.

\begin{corollary}
\label{t:hidden}
Under the assumptions of Proposition~\ref{t:HMMapproxPPkappaB},
for each $\epsy>0$  we can construct the approximating process 
$\bfPsi$ described in Section~\ref{HMMapprox} so that the 
the Markov processes $\bfPhi$ and $\bfPsi$ 
have unique invariant probability measures $\pi$ and $\varpi$,
respectively, satisfying, 
\[
\|\pi -\varpi\|_{v} \le\epsy.
\]
\qed
\end{corollary}

\newpage

\appendix
\begin{center}
\LARGE
\bf 
Appendix
\end{center}

\section{Appendix: Proof of \Proposition{DV4rlx}}

The drift condition (DV3) can be expressed as follows, in terms of the function $v = e^V$:
\[
\clD v\le (-\delta W +b\ind_C)v.
\]
By assumption, we have  $\delta W(x)\ge \delta$ everywhere.  Moreover,   $\delta W(x)\ge 1$ on the complement of the sublevel set $C_W(\delta^{-1}) $ (see \eqref{CFr}).  This set is compact under (A4),  so that the desired bound holds with,
\[
b_v \eqdef b\Bigl(\sup_{x\in C}v(x)\Bigr) 
	+ 
(1-\min(\delta,1)) 
	\Bigl(\sup_{x\in C_F(r)}v(x)\Bigr)  <\infty\, .
\]
This establishes part (i).

\medskip

Under \eqref{DV4rlx} we can apply  \Proposition{CT} (ii) with $h=v$, $f=  v$ and $g\equiv  b_v$ to obtain  
$
\alpha R_\alpha v +  R_\alpha v  \le     v   + \alpha^{-1} b_v$,  or 
\begin{equation}
\alpha  R_\alpha v\le (1+\alpha^{-1})^{-1} \bigl( v + \alpha^{-1}b_v \bigr). 
\label{RalphaDrift}
\end{equation}
Iterating this bound we obtain, for any $n\ge 1$,
\[
\begin{aligned}
(\alpha  R_\alpha )^nv&\le (1+\alpha^{-1})^{-n} v + 
 \alpha^{-1}b_v \sum_{k=1}^n (1+\alpha^{-1})^{-k  } 
\\
&\le v+  \alpha^{-1}b_v [1- (1+\alpha^{-1})]^{-1} = v+b_v.
\end{aligned}
\]
Hence $(\alpha  R_\alpha )^nv \le (1+b_v) v$, which is the first bound.  

The second follows from \eqref{RalphaDrift} and the (discrete-time)
comparison theorem of \cite{meyn-tweedie:book2}, which gives,
\[
\pi(v)<\infty \quad \text{\it and}\quad
\pi(v)\le (1+\alpha^{-1})^{-1}\bigl( \pi(v) +  \alpha^{-1}b_v\bigr). 
\] 
Rearranging terms gives $ (1+\alpha^{-1})\pi(v)\le    
\pi(v) +  \alpha^{-1}b_v$,
or $\pi(v)\le   b_v$ as claimed.
\qed

\section{Appendix: Proof of \Proposition{RapproxConverse}}
Under the separability assumption we can find, for each $n\ge 1$,  
a compact set $\ystate_n$ and a kernel $T_n$ supported on 
$\ystate_n$ satisfying  $\lll R-T_n\lll_v \le 2^{-n}$.   We assume without loss of generality that $\ystate_n\subseteq \ystate_{n+1}$ for each $n$, and that $\bigcup_n \ystate_n = \state$.

Writing $v_n = v \ind_{\ystate_n^c}$ we have  $\lll v_n \lll_v=1$   
and $T_n v_n \equiv 0$.  Consequently,   for each $n\ge 1$,
\begin{equation}
R v_n  =   (R-T_n)v_n + T_n v_n =  (R-T_n)v_n 
           \le  \lll R-T_n\lll_v v\le 2^{-n} v,\qquad n\ge 1.
\label{RvnBdd}
\end{equation}
The desired solution to (DV3) is constructed as follows.  First define the sequence of 
finite-valued functions
on $\state$,
\[
u_-^m \eqdef   v + \sum_{n=1}^m v_n
		=   \Bigl(1 +    \sum_{n=1}^m\ind_{\ystate_n^c}\Bigr)  v,\qquad v_-^m = R u_-^m,\quad m\ge 1,
\]
and denote $u_-=\lim_{m\to\infty} u_-^m$, $v_-=\lim_{m\to\infty} v_-^m$.  
Applying \eqref{RvnBdd}, we conclude that $v_-\in\Lv$, with 
the explicit bound,
\[
\| v_-\|_v  \le  \lll R\lll_v 
+ \sum_{n=1}^\infty \|Rv_n\|_v \le  \lll R\lll_v + 1.
\]
Each of the functions $v_-^m$ is continuous since $R$ has the strong Feller property.  These functions converge to $v_-$ uniformly on compact subsets of $\state$, showing that $v_-$ is continuous.  We let    $V_- = \log( v_-)$, which is also continuous. 
 

It follows from \Proposition{ResolveDR} that the resolvent equation holds,
$\clD v_- = v_- - u_-$, and consequently, 
recalling the nonlinear generator \eqref{fle78a}, 
\[
 \clH(V_-) = (v_-)^{-1} \clD v_-   =   1- u_-/v_-.
\]
By construction, the function $u_-/v_-$ has 
compact sublevel sets.
Writing $W = \max ( u_-/v_- -1 ,1)$ and $C=\{ x : W(x)\leq1 \}$ 
then gives,
\[
 \clH(V_-)    \le - W  +  2\ind_C,
\]
which is a version of (DV3).  The function $W$ is not continuous. 
However, it has compact sublevel sets, so there exists a 
continuous function $W_-\colon\state\to[1,\infty)$ 
with compact sublevel sets, satisfying $W_-\le W$ 
everywhere. The pair $(V_-,W_-)$ is the 
desired solution to (DV3).
\qed

\section{Proof of \Proposition{i-ii}}

Before giving the proof, we state and prove some
preliminary results. The assumptions of \Proposition{i-ii}
remain in effect throughout this subsection.

On setting $h=v=e^V$, $f=\delta W h$, 
and $g=b\ind_C h$ in \Proposition{CT} 
we obtain the following bound:

\begin{lemma}
\label{t:W1V}
Under {\em (DV3)}, with $v=e^V$, 
we have, 
\begin{equation*}
R I_{W_1}   v \le  v   +  b s_0,
\end{equation*}
where $W_1= 1+\delta W$, $s_0 = R I_C v$
and for any function $F$, 
$I_F$ denotes the multiplication kernel 
$I_F(x,dy)=F(x)\delta_x(dy)$.
\end{lemma}
 
For each $r\ge 1$, we define the compact sets,
\begin{equation*}
C_r = C_v(r)\cap C_W(r), 
\end{equation*} 
in the notation of equation \eqref{CFr}.  
From the assumption that $V$ and $W$ are 
continuous with compact sublevel sets,
we obtain,
\begin{equation}
\lim_{r\to\infty} \inf_{x\in C_r^c} V(x) 
= \lim_{r\to\infty} \inf_{x\in C_r^c} W(x) = \infty.
\label{CrCoerce}
\end{equation}
The above bounds on the resolvent will allow us 
to approximate $R$ by a kernel 
supported on $C_r$ for suitably large $r\ge 1$.  
To that end, we choose and fix a continuous   
function  $W_0\colon\state\to[1,\infty)$   
in $L_\infty^W$, satisfying $\|W_0^2\|_W =1$, and  
whose growth  at infinity is strictly slower than 
$W^\half$  in the sense that,
\begin{equation}
\lim_{r\to\infty} \| W_0^2\ind_{C_W(r)^c}\|_W = 0\, .
\label{Wo}
\end{equation}
This can be equivalently expressed,
\[
\lim_{r\to\infty} \sup_{x\in\state}\,
        \Bigl [
        \frac{W_0(x)}{\sqrt{W(x)}}\,\ind_{\{W(x)>r\}}
        \Bigr ]
= 0\,.
\]
The weighting function is simultaneously increased to,
\begin{equation*}
v_0 = W_0 v.
\end{equation*}

The following Lemma justifies truncating $R$
to a compact set.

\begin{lemma}
\label{t:RbddW}
Under {\em (DV3)} the resolvent kernel $R$ satisfies   
$\lll R I_{W}\lll_v<\infty$, and,
\begin{equation}
\lim_{r\to\infty}  \lll I_{W_0} (R - I_{C_r} R I_{C_r} )I_{W_0}\lll_{v_0} = 0.
\label{WoLimit}
\end{equation}
\end{lemma}

\begin{proof} 
\Lemma{W1V} implies that $\lll R I_{W}\lll_v$ is finite as claimed:  
We have the explicit bound  
$\lll R I_{W}\lll_v\le \delta^{-1}(1+b \|s_0\|_v)$.   
The limit \eqref{WoLimit} is also based on the same lemma.
Starting with the identity $R - I_{C_r} R I_{C_r}  
=    I_{C_r} R I_{C^c_r} +  I_{C^c_r} R$ we obtain,
\begin{equation}
\begin{aligned}
I_{W_0}  (R - I_{C_r} R I_{C_r} )I_{W_0}  v_0    
&  =  I_{W_0} (R - I_{C_r} R I_{C_r} )I_{W^2_0}  v     
\\
	&  =  I_{W_0}  [  I_{C_r} R I_{C^c_r} I_{W^2_0}] v   
		+ I_{W_0}  [    I_{C^c_r} RI_{W_0}   ]v_0. 	     
\end{aligned}
\label{WoLimitA}
\end{equation}
These two terms can be bounded separately.  For the first
term on the right-hand-side consider the following,
\[
 [ I_{C_r} R I_{C^c_r} I_{W^2_0}] v 
 \le
 \lll     I_{C_r} R I_{C^c_r} I_{W} \lll_v  \epsy_r v
 \le
   \lll    R  I_{W} \lll_v  \epsy_r  v,
\]
where $\epsy_r\eqdef \sup_{x\in C_r^c} W_0(x) W^{-\half}(x)$.   
Multiplying both sides by $W_0$ then gives,
\[
 [ I_{W_0}I_{C_r} R I_{C^c_r} I_{W_0}] v_0 
 \le
   \lll    R  I_{W} \lll_v  \epsy_r  v_0,\qquad r\ge 1,
\]
which means that $\lll I_{W_0}I_{C_r} R I_{C^c_r} I_{W_0} \lll_{v_0}\le    \lll    R  I_{W} \lll_v  \epsy_r $ for each $r$.

Bounds on the second term in \eqref{WoLimitA} are obtained similarly 
through a second truncation. Write, for any $n\ge 1$, 
\[
 [    I_{C_r^c} RI_{W^2_0}   ]v 
=  [    I_{C_r^c} RI_{W^2_0}  \ind_{C_n} ]v 
+ [    I_{C_r^c} RI_{W^2_0}  \ind_{C_n^c} ]v .
\]
Arguing as above we have $\lll    I_{C_r^c} RI_{W^2_0}  
\ind_{C_n^c} \lll_v \le   \lll    R  I_{W} \lll_v  \epsy_n$.  
Moreover, $W^2_0 v \le W v\le n^2$ on $C_n$, which gives,
\[
 [    I_{C_r^c} RI_{W^2_0}   ]v 
\le n^2     I_{C_r^c}   +    \lll    R  I_{W} \lll_v  \epsy_n v .
\]
Multiplying both sides of this equation by $W_0$ gives,
\[
 [  I_{W_0}  I_{C_r^c} RI_{W_0}   ]v_0
\le     n^2   I_{C_r^c}  W_0  +    \lll    R  I_{W} \lll_v  \epsy_n v_0,
\]
so that 
\begin{equation}
\lll  I_{W_0}  I_{C_r^c} RI_{W_0} \lll_{v_0}
\le 
 n^2   \lll I_{C_r^c}  W_0\lll_{v_0} +  \lll    R  I_{W} \lll_v  \epsy_n.
\label{WoRoW}
\end{equation} 
And also,
\[
\lll I_{C_r^c}  W_0\lll_{v_0} = \sup_{x\in C_r^c} \frac{W_0(x)}{v_0(x)} 
= \sup_{x\in C_r^c} \frac{1}{v(x)} \le \frac{1}{r} .
\] 
This combined with \eqref{CrCoerce} implies that \eqref{WoRoW}  can be made 
arbitrarily small by choosing large $n$ and then large $r$.
\end{proof}

\begin{lemma}
\label{t:barRbddW}
Under~{\em (A1)} and~{\em (A2)}, for each $r\ge 1$ and $\epsy>0$,  
there exists $t_0>0$ and $t_1<\infty$ in the 
definition \eqref{e:Rhat} such that,
\begin{equation*}
\lll I_{W_0}I_{C_r}  (R -  \barR ) I_{C_r}I_{W_0}\lll_{v_0} \le \epsy,
\end{equation*}
where $\barR=\barR_1$.
\end{lemma}

\begin{proof}
Since $W_0$ and $v_0$ are bounded on $C_r$,  we can apply the bound,
\[
  \lll I_{W_0}I_{C_r}  (R -  \barR ) I_{C_r}I_{W_0}\lll_{v_0}  \le \bigl( \sup_{x\in C_r}  W_0(x)^2 v_0(x)  \bigr) 
  \lll  I_{C_r}  (R -  \barR ) I_{C_r} \lll_1
\]
Hence
 it is sufficient to prove the result with $W_0=v_0=1$.  
 
We have by definition of $\barR$,
\[
  \lll  I_{C_r}  (R -  \barR ) I_{C_r} \lll_1  = \sup_{x\in C_r} \int_{t\in [t_0,t_1]^c} e^{-t} P^t(x,C_r)\, dt.
\]
The right hand side is bounded by $t_0 + e^{-t_1}$,  which can be made arbitrarily small by choice of $t_0>0$ and $t_1<\infty$.
\end{proof}

\Proposition{i-ii} will be seen as a corollary to the following more general
bound:

\begin{proposition}
\label{t:finiteRank}
For any $\epsy>0$ there exists 
a finite-rank kernel $T$ satisfying:
\begin{equation*}
\lll I_{W_0} [R - T] I_{W_0}\lll_{v_0} \le \epsy.
\end{equation*}
The kernel can be taken 
of the form,
\begin{equation}
 T=\sum_{ij}r_{ij}\IND_{C_i}\otimes\nu_j
\label{e:Tsimple}
\end{equation}
where $\{C_i:1\leq i\leq N\}$ is
a finite collection of disjoint, open, precompact sets,  $\{r_{ij}\}$ are non-negative constants, and
$\{\nu_i\}$ are probability measures on $(\state,\clB)$
with each $\nu_i$ supported on $C_i$.
\end{proposition}

\begin{proof}
\Lemma{RbddW} and
\Lemma{barRbddW} imply 
that for any $\epsy>0$ we can find $r_0\ge 1$ such that,
$$ 
	\lll I_{W_0} (R - I_{C_{r_0}} \barR I_{C_{r_0}} )  I_{W_0}\lll_{v_0} \le \epsy/2.
$$
With this value of $r_0$ fixed, note that (A2)~implies that 
for any $\epsy_0>0$  we can construct a kernel $T (x,dy) = t(x,y) dy$ 
of the form given in \eqref{e:K} such that $|t(x,y)-\barxi(x,y)|\le \epsy_0$ 
for $(x,y)\in C_{r_0}\times C_{r_0}$ (see definition of $\barxi$ above \Proposition{StrongFeller}).    In particular, the functions 
$\{s_i\}$ and the densities of the $\nu_i$ can be taken as indicator 
functions, so that this   is simply the approximation of the continuous 
function $\barxi(\varble,\varble)$ by simple functions. 
Consequently,
\[
\begin{aligned}
\lll I_{W_0} [R - T ] I_{W_0}\lll_{v_0}   
& 
\le \epsy/2 +  \lll I_{W_0} I_{C_{r_0}} [R - T ] I_{C_{r_0}} I_{W_0}\lll_{v_0}
\\
& 
\le \epsy/2 +     \epsy_0\sup_{x\in C_{r_0} } W_0(x) \sup_{x\in C_{r_0} } \bigl( v(x) W_0^2(x) \bigr) \muLeb(C_{r_0})   
\\
& 
\le  \epsy/2 +     \epsy_0  r_0^4 \muLeb(C_{r_0})   
,\end{aligned}
\]
where $\muLeb(C_{r_0})   $ denotes the Lebesgue measure of the 
bounded set $C_{r_0}$.
The right-hand-side is bounded by $\epsy$ on choosing
 $\epsy_0 =  [r_0^4 \muLeb(C_{r_0}) ]^{-1} (\epsy/2 ) $.
\end{proof}

\medskip

\noindent
{\em Proof of \Proposition{i-ii}. }
Since \Proposition{finiteRank} was proved
for an arbitrary function $W_0$ satisfying 
(\ref{Wo}), we can take $W_0$ equal to 
a constant, say $w\geq 1$. 
First consider the case $\kappa=1$. There,
applying \Proposition{finiteRank} with $\epsilon/2$
instead of $\epsilon$, we obtain a finite-rank
kernel of the form \eqref{e:Tsimple}.
Letting $\clE_0=\kappa[-I+T]$,
$$\lll\clD_\kappa-\clE_0\lll_v
=\lll \kappa[\kappa R_\kappa-T]\lll_v
=\lll R-T\lll_v
=\frac{1}{w}\lll I_{W_0}[R-T]I_{W_0}\lll_{v_0}
\leq \epsilon/2.
$$
Now we define $\clE=\clE_0+ \IND_{C_0}\otimes\nu_1$, with $C_0 =\state\setminus\cup_{1\leq i\leq N} C_i$ and $\nu_1$ a probability measure supported on $C_1$.
We have $\clE 1\equiv 0$ as required, and the following bound holds:
\[
\lll\clD_\kappa-\clE\lll_v \le \lll\clD_\kappa-\clE_0\lll_v + \nu_1(v)\Bigl(\sup_{x\in C_0} \frac{1}{v(x)}\Bigr).
\]
Recall that  $C_0^c  =\bigcup_{i\ge 1} C_i$.  If the $\{C_i: i\ge 1\}$ are constructed so that $C_0^c\subset C_v(r)$,  then  
the right hand side is bounded by $ \nu_1(v) r^{-1}$.  For $r>0$ sufficiently large, this is less than $\epsy$, as required.

For a fixed, general $\kappa$ we consider the scaled 
process $\{Z(t):=\Phi(t/\kappa)\;:\;t\geq 0\}$
and note it satisfies exactly the same assumptions 
as $\{\Phi(t)\}$. Also,  $\kappa R_\kappa$,
is the  resolvent
kernel for $\{Z(t)\}$ (corresponding to the
parameter $\alpha=1$)
so that, as before by \Proposition{finiteRank},
we obtain the required bound.
\qed


\bibliographystyle{plain}

\begin{thebibliography}{10}


\bibitem{balaji-meyn}
S.~Balaji and S.P. Meyn.
\newblock Multiplicative ergodicity and large deviations for an irreducible
  {M}arkov chain.
\newblock {\em Stochastic Process. Appl.}, 90(1):123--144, 2000.

\bibitem{busvlisch12}
A.~Bu{\v s}i{\'c}, I.~Vliegen, and A.~Scheller-Wolf.
\newblock Comparing {M}arkov chains: Aggregation and precedence relations
  applied to sets of states, with applications to assemble-to-order systems.
\newblock 37(2):259--287, 2012.

\bibitem{denmehmey11}
K.~Deng, P.~Mehta, and S.~Meyn.
\newblock Optimal {Kullback-Leibler} aggregation via spectral theory of
  {Markov} chains.
\newblock 56(12):2793 --2808, Dec. 2011.



\bibitem{dembo-zeitouni:book}
A.~Dembo and O.~Zeitouni.
\newblock {\em Large Deviations Techniques and Applications}.
\newblock Springer-Verlag, New York, second edition, 1998.

\bibitem{donsker-varadhan:I-II}
M.D. Donsker and S.R.S. Varadhan.
\newblock Asymptotic evaluation of certain {M}arkov process expectations for
  large time. {I}. {I}{I}.
\newblock {\em Comm. Pure Appl. Math.}, 28:1--47; ibid. 28:279--301, 1975.

\bibitem{donsker-varadhan:III}
M.D. Donsker and S.R.S. Varadhan.
\newblock Asymptotic evaluation of certain {M}arkov process expectations for
  large time. {I}{I}{I}.
\newblock {\em Comm. Pure Appl. Math.}, 29(4):389--461, 1976.

\bibitem{donsker-varadhan:IV}
M.D. Donsker and S.R.S. Varadhan.
\newblock Asymptotic evaluation of certain {M}arkov process expectations for
  large time. {I}{V}.
\newblock {\em Comm. Pure Appl. Math.}, 36(2):183--212, 1983.

\bibitem{dowmeytwe95a}
D.~Down, S.P. Meyn, and R.L. Tweedie.
\newblock Exponential and uniform ergodicity of {M}arkov processes.
\newblock {\em Ann. Probab.}, 23(4):1671--1691, 1995.

\bibitem{ethkur86}
S.N. Ethier and T.G. Kurtz.
\newblock {\em Markov Processes : Characterization and Convergence}.
\newblock John Wiley \& Sons, New York, 1986.

\bibitem{fen99a}
J.~Feng.
\newblock Martingale problems for large deviations of {Markov} processes.
\newblock {\em Stochastic Process. Appl.}, 81:165--212, 1999.

\bibitem{feng-kurtz:book}
J.~Feng and T.G. Kurtz.
\newblock {\em Large deviations for stochastic processes}, volume 131 of {\em
  Mathematical Surveys and Monographs}.
\newblock American Mathematical Society, Providence, RI, 2006.

\bibitem{fle78a}
W.H. Fleming.
\newblock Exit probabilities and optimal stochastic control.
\newblock {\em App. Math. Optim.}, 4:329--346, 1978.



 
\bibitem{gaugou05}
E.~Gaussier and C.~Goutte.
\newblock Relation between PLSA and NMF and implications.
\newblock In {\em SIGIR '05: Proceedings of the 28th annual international ACM
SIGIR conference on Research and development in information retrieval}, pages
601--602, New York, NY, USA, 2005. ACM.




\bibitem{gonwu06a}
F.Z. Gong and L.M. Wu.
\newblock Spectral gap of positive operators and applications.
\newblock {\em J. Math. Pures Appl.}, 85:151--191, 2006.

\bibitem{guileowuyao09}
A.~Guillin, C.~L\'{e}onard, L.~Wu, and N.~Yao.
\newblock Transportation-information inequalities for {Markov} processes.
\newblock 144(3):669--695, July 2009.

\bibitem{hof01}
T.~Hofmann.
\newblock Unsupervised learning by {Probabilistic Latent Semantic Analysis}.
\newblock {\em Mach. Learn.}, 42(1-2):177--196, 2001.

\bibitem{kar85a}
N.V. Kartashov.
\newblock Criteria for uniform ergodicity and strong stability of {Markov}
  chains with a common phase space.
\newblock {\em Theor. Probability Appl.}, 30:71--89, 1985.

\bibitem{kar85b}
N.V. Kartashov.
\newblock Inequalities in theorems of ergodicity and stability for {Markov}
  chains with a common phase space.
\newblock {\em Theor. Probability Appl.}, 30:247--259, 1985.

\bibitem{kontoyiannis-meyn:I}
I.~Kontoyiannis and S.P. Meyn.
\newblock Spectral theory and limit theorems for geometrically ergodic {M}arkov
  processes.
\newblock {\em Ann. Appl. Probab.}, 13:304--362, February 2003.

\bibitem{kontoyiannis-meyn:II}
I.~Kontoyiannis and S.P. Meyn.
\newblock Large deviation asymptotics and the spectral theory of
  multiplicatively regular {M}arkov processes.
\newblock {\em Electron. J. Probab.}, 10(3):61--123, 2005.

%

\bibitem{lob70a}
C.~Lobry.
\newblock Contr\^olabilit\'e des syst\`emes non lin\'eaires.
\newblock {\em SIAM J. Control}, 8:573--605, 1970.

\bibitem{meyn-tweedie:book2}
S.~P. Meyn and R.~L. Tweedie.
\newblock {\em {Markov Chains and Stochastic Stability}}.
\newblock Cambridge University Press, London, 2nd edition, 2009.
\newblock Published in the Cambridge Mathematical Library. 1993 edition online:
  {\tt http://black.csl.uiuc.edu/{\~{ }}meyn/pages/book.html}.

\bibitem{meytwe93e}
S.P. Meyn and R.L. Tweedie.
\newblock {Generalized resolvents and {Harris} recurrence of {Markov}
  processes}.
\newblock {\em Contemporary Mathematics}, 149:227--250, 1993.

\bibitem{meytwe93a}
S.P. Meyn and R.L. Tweedie.
\newblock Stability of {Markovian} processes {II}: Continuous time processes
  and sampled chains.
\newblock {\em Ann. Appl. Probab.}, 25:487--517, 1993.

\bibitem{meytwe93b}
S.P. Meyn and R.L. Tweedie.
\newblock Stability of {Markovian} processes {III}: {Foster}-{Lyapunov}
  criteria for continuous time processes.
\newblock {\em Ann. Appl. Probab.}, 25:518--548, 1993.

\bibitem{nummelin:book}
E.~Nummelin.
\newblock {\em General Irreducible {M}arkov Chains and Nonnegative Operators}.
\newblock Cambridge University Press, Cambridge, 1984.

\bibitem{polcinmen12}
S.~Polidoro, C.~Cinti, and S.~Menozzi.
\newblock Two-sided bounds for degenerate processes with densities supported in
  subsets of $\Re^n$.
\newblock {\em arXiv preprint arXiv:1203.4918}, 2012.



\bibitem{reytho01a}
L.~Rey-Bellet and L.~E. Thomas.
\newblock Fluctuations of the entropy production in anharmonic chains.
\newblock {\em Ann. Inst. Henri Poincar\'e}, 3(3):483--502, 2002.

\bibitem{rogers-williams:I}
L.~C.~G. Rogers and D.~Williams.
\newblock {\em Diffusions, {M}arkov processes, and martingales. {V}ol. 1}.
\newblock Cambridge Mathematical Library. Cambridge University Press,
  Cambridge, 2000.
\newblock Foundations, Reprint of the second (1994) edition.



\bibitem{rogers-williams:II}
L.C.G. Rogers and D.~Williams.
\newblock {\em Diffusions, {M}arkov processes, and martingales. {V}ol. 2}.
\newblock Cambridge University Press, Cambridge, 2000.


\bibitem{sharajsma08}
M.~Shashanka, B.~Raj, and P.~Smaragdis.
\newblock {Probabilistic Latent Variable Models} as nonnegative factorizations.
\newblock {\em Computational Intelligence and Neuroscience}, pages 1--8, 2008.



\bibitem{strvar72a}
D.W. Stroock and S.R. Varadhan.
\newblock On the support of diffusion processes with applications to the strong
  maximum principle.
\newblock In {\em Proceedings of the 6th Berkeley Symposium on Mathematical
  Statistics and Probability}, pages 333--368. University of California Press,
  1972.

\bibitem{susjur72}
H.J. Sussmann and V.~Jurdjevic.
\newblock Controllability of nonlinear systems.
\newblock {\em J. Differential Equations}, 12:95--116, 1972.

\bibitem{vei69}
A.F. Veinott~Jr.
\newblock Discrete dynamic programming with sensitive discount optimality
  criteria.
\newblock {\em Ann. Math. Statist.}, 40(5):1635--1660, 1969.

\bibitem{wu95a}
L.M. Wu.
\newblock Large deviations for {Markov processes} under superboundedness.
\newblock {\em C. R. Acad. Sci Paris S\'erie I}, 324:777--782, 1995.

\bibitem{wu:01}
L.M. Wu.
\newblock Large and moderate deviations and exponential convergence for
  stochastic damping {H}amiltonian systems.
\newblock {\em Stochastic Process. Appl.}, 91(2):205--238, 2001.


\bibitem{wu04}
L.~Wu.
\newblock Essential spectral radius for {M}arkov semigroups. {I}. {D}iscrete
  time case.
\newblock {\em Prob. Theory Related Fields}, 128(2):255--321, 2004.

\end{thebibliography}

\end{document}